\documentclass[a4paper]{amsart}

\usepackage{verbatim,amsmath,amssymb,enumitem,amsthm}
\usepackage[displaymath,pagewise]{lineno} 
\usepackage{mathtools}
\usepackage{xcolor}

\newtheorem{thm}{Theorem}
\newtheorem{prop}[thm]{Proposition}
\newtheorem{lem}[thm]{Lemma}
\newtheorem{cor}[thm]{Corollary}

\newtheorem{rem}[thm]{Remark}

\theoremstyle{remark}
\newtheorem{remark}[thm]{Remark}        

\numberwithin{equation}{section} \numberwithin{thm}{section}

\newtheorem{definition}[thm]{Definition}

\newcommand*\patchAmsMathEnvironmentForLineno[1]{%
  \expandafter\let\csname old#1\expandafter\endcsname\csname #1\endcsname
  \expandafter\let\csname oldend#1\expandafter\endcsname\csname end#1\endcsname
  \renewenvironment{#1}%
     {\linenomath\csname old#1\endcsname}%
     {\csname oldend#1\endcsname\endlinenomath}}%
\newcommand*\patchBothAmsMathEnvironmentsForLineno[1]{%
  \patchAmsMathEnvironmentForLineno{#1}%
  \patchAmsMathEnvironmentForLineno{#1*}}%
\AtBeginDocument{%
\patchBothAmsMathEnvironmentsForLineno{equation}%
\patchBothAmsMathEnvironmentsForLineno{align}%
\patchBothAmsMathEnvironmentsForLineno{flalign}%
\patchBothAmsMathEnvironmentsForLineno{alignat}%
\patchBothAmsMathEnvironmentsForLineno{gather}%
\patchBothAmsMathEnvironmentsForLineno{multline}%
}
\newcommand{\vf}{\varphi}
\newcommand{\cal}{\mathcal}
\newcommand{\PR}{{P\! R}}

\newcommand{\UPR}{{\cal U}_{\rm \PR}}
\newcommand{\dist}{{\rm dist}\,}

\newcommand{\sph}{{{\mathbb S}}}
\newcommand{\sphd}{{{\mathbb S}^{d-1}}}
\newcommand{\R}{{\mathbb R}}
\newcommand{\N}{{\mathcal N}}

\newcommand{\nor}{{\rm Nor}\,}
\newcommand{\supp}{{\rm supp}\,}

\newcommand{\Tan}{{\rm Tan}\,}

\newcommand{\reach}{{\rm reach}\,}

\newcommand{\subgr}{{\rm subgr}\,}

\newcommand{\spa}{\operatorname{span}}

\newcommand{\eps}{\varepsilon}

\newcommand{\sS}{\cal{S}}
\newcommand{\sM}{\cal{M}}

\def\en{\mathbb N}
\def\er{\mathbb R}

\def\C{\mathcal C}
\def\K{\mathcal K}
\def\H{\mathcal H}

\def\T{\mathcal T}
\def\A{\mathcal A}
\def\G{\cal G}
\def\D{\cal D}
\def\U{\cal U}
\def\ii{\mathfrak i}
\def\cO{\mathcal O}
\def\bM{\mathbf M}
\def\cM{\mathcal M}
\def\MB{\mathcal {MB}}

\newcommand{\graph}{\operatorname{graph}}

\newcommand{\co}{\operatorname{co}}

\newcommand{\epi}{\operatorname{epi}}
\newcommand{\hyp}{\operatorname{hyp}}

\newcommand{\spt}{\operatorname{spt}}

\newcommand{\WDC}{{\rm WDC}}
\newcommand{\UWDC}{{\cal U}_{\WDC}}
\newcommand{\UWDCG}{{\cal U}_{\WDC}^{\cal G}}

\newcommand{\llc}{\;\halfsq\;}

\def\halfsq{\hbox{\kern1pt\vrule height 7pt\vrule width6pt height 0.4pt depth0pt\kern1pt}}
\def\ihalfsq{\hbox{\kern1pt \vrule width6pt height 0.4pt depth0pt
                   \vrule height 7pt \kern1pt}}

\begin{document}

\title{Curvatures for unions of WDC sets}
\author{Du\v san Pokorn\'y}
\thanks{The research was supported by GA\v CR~15-08218S and GA\v CR~18-11058S}

\begin{abstract} 
	We prove the existence of the curvature measures for a class of $\UWDC$ sets, which is a direct generalization of $\UPR$ sets studied by Rataj and Z\"ahle. Moreover, we provide a simple characterisation of $\UWDC$ sets in $\er^2$ and prove that in $\er^2$ the class of $\UWDC$ sets contains essentially all classes of sets known to admit curvature measures.
\end{abstract}

\email{dpokorny@karlin.mff.cuni.cz}

\date{\today}
\maketitle

\section{Introduction}\label{sec:intoduction}

One of the important tasks of modern curvature theory is to extend the notion of curvature to sets with singularities beyond convex sets. This has been done by Federer by constructing the curvature measures for the sets of positive reach (\cite{Fe59}), by Fu in the case of subanalytic sets (\cite{Fu94}), by Z\"ahle and Rataj for certain locally finite unions of sets of positive reach called $\UPR$ sets (\cite{RZ01}) and also for the so called Lipschitz manifolds of bounded curvature (\cite{RZ03},\cite{RZ05}). 
Recently, the existence of the curvature measures has been proven for the class of (locally) $\WDC$ sets (\cite{PR13},\cite{FPR15}).

The aim of the present paper is to extend the curvature theory to the class of $\UWDC$ sets, which is formed by special locally finite unions of $\WDC$ sets (see Definition~\ref{def:UWDC}) that relates to $\WDC$ sets the same way $\UPR$ sets relate to the sets of positive reach.
Since $\WDC$ sets are a direct generalisation of the sets of positive reach, $\UWDC$ sets are a generalisation of $\UPR$ sets.

The plan of the paper is the following.
First we recall some basic definitions and some facts mostly about $\WDC$ sets (Section~\ref{sec:preliminaries}). Next we prove our first main result, the existence of the normal cycle for $\UWDC$ sets (Theorem~\ref{thm:UWDChasNC}) and we also prove the Kinematic Formula for the corresponding curvature measures (Theorem~\ref{thm:kinematicFormula}).
In Section~\ref{sec:inPlane} we prove our second main theorem that provides a geometric description of $\UWDC$ sets in $\er^2$ (Theorem~\ref{T:characterizationUWDC}). The main part of the theorem (equivalence $\ref{cond:MisUWDC}\!\!\iff\!\!\ref{cond:MdcBoundary}$) says that a compact set $M\subset\er^2$ is $\UWDC$ if and only if its complement has finitely many connected components and $\partial M$ is a union of finitely many DC graphs (see Section~\ref{subs:LipADCgRAPHS} for the definition). 
In the last section we add some other observations about $\UWDC$ sets in plane, mainly that every compact subanalytic set in $\er^2$ is $\UWDC$ and also that compact Lipschitz manifolds of bounded curvature of Rataj and Z\"ahle are also $\UWDC$ (and actually even $\WDC$). We in fact believe that in $\er^2$ the class of $\UWDC$ sets is the maximal integral geometric class (in the sense of \cite{FPR15}, cf. also \cite{Fu17}).

\section{Preliminaries}\label{sec:preliminaries}

\subsection{Notation and basic definitions}\color{black}
We will use the notation $A^c$ for the complement of a set $A$. 
In any vector space $V$, we use the symbol $0$ for the zero element and $\spa M$ for the linear span of a set $M$.
By a subspace of $V$ we always mean a linear subspace, unless specified otherwise. The symbol $U(x,r)$ ($ B(x,r)$) denotes the open (closed) ball with centre $x$ and radius $r>0$ (in $\er^d$).
$\Tan(A,a)$ denotes the tangent cone of $A\subset X$ at $a\in X$ ($u\in\Tan(A,a)$ if and only if $u=\lim_{i\to\infty}r_i(a_i-a)$ for some $r_i>0$ and $a_i\in A\setminus\{a\}$, $a_i\to a$).
For a convex set $K\subset\er^d$ the symbol $\nor(K,x)$ denotes the unit normal cone of $K$ at $x$.


We shall work mostly in the Euclidean space $\R^d$ with the standard scalar product $u\cdot v$ and norm $|u|$, $u,v\in\R^d$. The unit sphere in $\R^d$ will be denoted by $\sph^{d-1}$. We denote by $\Pi_V$  the orthogonal projection to $V$.
The angle between two vectors $v,w\in\sph^{d-1}$ (defined as usual by the formula $\arctan(v\cdot w)$) will be denoted $\rho(v,w)$.
The set of all Euclidean motions on $\er^d$ will be denoted $\G_d$ and the unique Haar measure on $\G_d$ will be denoted $\gamma_d$. 
For $t\in\er$ and $v\in\sphd$, $H_{v,t}$ will denote the halfspace in $\er^d$ defined by $\{y\in\er^d: y\cdot v\leq t \}$.
For $A\subset\er^d$ and $\eps>0$ we define the set $A_\eps\coloneqq\{y\in\er^d: |x-y|\leq \eps\}$ and call it the parallel set of $A$ (with a radius $\eps$).
If~$K,M\subseteq\er^d$ are non-empty compact sets we denote by $\dist_{\H}(K,M)$ the Hausdorff distance between $K$ and $M$. Recall that for $\eps>0$
\begin{equation}\label{eq:hausdorffDistance}
\dist_{\H}(K,M)\leq\eps \iff M\subseteq K_\eps\text{ and } K\subseteq M_\eps.
\end{equation}

A mapping is called $K$-Lipschitz if it is Lipschitz with a constant $K$.
If $H$ is a finite-dimensional Hilbert space, $U\subseteq H$ open, $f:U\to\R$ locally Lipschitz and $x\in U$, we denote by $\partial f(x)$ the {\it Clarke subdifferential of $f$ at $x$}, which can be defined as the closed convex hull of all limits $\lim_{i\to\infty}f'(x_i)$ such that $x_i\to x$ and $f'(x_i)$ exists for all $i\in\en$ (see \cite[\S1.1.2]{C}). Since we identify $H^*$ with $H$ in the standard way, we sometimes consider $\partial f(x)$ as a subset of $H$.
For a real function $f$ defined on a neighbourhood of a point $x\in\er$ the symbols $f'_+(x)$ and $f_-'(x)$ will denote the one-sided derivatives at $x$ from  the right and the left, respectively. 
We say that a mapping $f:[a,b]\to\er^d$ is $\C^2$ if there is a $\C^d$ mapping $\tilde f:(c,d)\to\er^d$ for some $(c,d)\supset[a,b]$ such that $\tilde f|_{[a,b]}=f$.

For $n\in\en$ we denote by $\Sigma_n$ the system of all nonempty subsets of $\{1,\dots,n\}$ and we put $\Sigma_n^0=\Sigma_n\cup\{\varnothing\}$.
The cardinality of a set $A$ will be denoted by $|A|$.

For a curve $\gamma:[a,b]\to\er^d$ we denote the image of $\gamma$ by $\Im(\gamma)$ (i.e. $\Im(\gamma)=\gamma([a,b])$). 

The symbol $\chi(A)$ will denote the Euler-Poincar\'e characteristic of a set $A$.

\subsection{Legendrian and normal cycles}  \label{Legendrian}\color{black}
We follow the notation and terminology from the Federer's book \cite{Fe69}.
Given an open subset $U$ of $\er^d$ and $0\leq k\leq d$ an integer, let ${\bf I}_k(U)$ denote the space of $k$-dimensional integer multiplicity rectifiable currents in $U$. Each current $T\in{\bf I}_k(U)$ can be represented by integration as
\begin{equation} \label{current}
T=(\H^{k}\llc W(T))\wedge \iota_Ta_T,
\end{equation}
where $W(T)$ is a $(\H^{k},k)$-rectifiable subset of $U$ (``carrier'' of $T$), $a_T$ is a unit simple tangent $k$-vectorfield of $W(T)$ and $\iota_T$ is an integer-valued integrable function over $W(T)$ (``index function'') associated with $T$. Note that the carrier $W(T)$ is not uniquely determined and need not be closed, in contrast with the support $\spt T$ which is closed by definition.

The mass norm $\bM(T)$ of a current $T$ is defined as the supremum of values $T(\phi)$ over all differential forms $\phi$ with $|\phi|\leq 1$. 

\begin{definition} \rm
	A {\it Legendrian cycle} is an integer multiplicity rectifiable $(d-1)$-current $T\in{\bf I}_{d-1}(\er^d\times \sph^{d-1})$ with the properties:
	\begin{eqnarray}
	&&\partial T=0\quad (T\text{ is a cycle}),\\
	&&T\llc\alpha=0\quad (T\text{ is Legendrian}),
	\end{eqnarray}
	where $\alpha$ is the contact $1$-form in $\er^d$ acting as 
	$\langle (u,v),\alpha(x,n)\rangle=u\cdot n$ (cf.\ \cite{Fu94}). 
\end{definition}

%

Let $v\in\sphd$ and $t\in\er$ be given. We shall say that the current $T$ {\it touches} the halfspace $H_{v,t}$ (or, equivalently, that $H_{v,t}$ touches $T$) if there exists a point $x\in\er^d$ such that $(x,-v)\in\spt T$ and $x\cdot v=t$.

\begin{definition}{(\cite[Definition~4.2]{PR13})}  \label{D-NC}  \rm
	We say that a compact set $A\subset\er^d$ admits a {\it normal cycle} $T$ if $T$ is a Legendrian cycle satisfying \begin{equation}\label{eq:NormalCycleTouching}
	\text{almost all halfspaces do not touch }T
	\end{equation} 
	and
\begin{equation}\label{eq:NormalCycleEuler}
\langle T,\pi_1,-v\rangle(H_{v,t}\times \sph^{d-1})=\chi(M\cap H_{v,t})\; \text{for $\H^d$-almost all $(v,t)\in\sph^{d-1}\times\er$}.
\end{equation}
	Such a $T$ is then unique (see \cite[Theorem~3.2]{Fu94} and \cite[Lemma~4.4]{PR13}), we write $T=N_A$ and call it the {\it normal cycle of} $A$.
\end{definition}\color{black}

\begin{remark} \rm  \label{rem-uniq}
There are various classes of sets known to admit the normal cycle, such as (compact) sets with positive reach \cite{Z86}, ${\mathcal U}_{\operatorname{PR}}$ sets defined in \cite{RZ01}, subanalytic sets (see \cite{Fu94}), the so-called Lipschitz manifolds with locally bounded inner curvature $\MB_d$ defined and studied in \cite{RZ03} and \cite{RZ05}, or (locally) $\WDC$ sets defined in \cite{PR13} (see Section~\ref{S_WDC} for more details).
\end{remark}

We will not need the exact definitions of subanalytic sets or of the class $\MB_d$, we will only use their following (well known) properties.
We start with two definitions.

\begin{definition}{(cf. also \cite[Conjecture~5.2]{FPR15})}
For $d\in\en$ we define the class $\mathcal N(\er^d)$ of compact sets $A\subset\er^d$ with the property that
there exists a monotone sequence $M_1\supset M_2 \supset\cdots $ of compact $\C^2$-smooth domains $\bigcap_n M_n=  A$, where the masses of $N_{M_n}$ are bounded by a fixed constant.	
\end{definition}

\begin{definition}
	Suppose that $M\subset\er^d$ admits the normal cycle $N_M$ with the corresponding index function $\ii_M$.
	We will say that $M$ satisfies condition (I) if $\ii_M$ has the following property: if 
	\begin{equation*}
	U\cap  \partial M \subset K\quad\text{and}\quad x\in U\cap \partial M \cap \partial K
	\end{equation*}
	for some $U\subset \er^d$ non-empty open and $K$ convex,
	then 
	\begin{equation*}
	\ii_M(x,v)\not =0
	\end{equation*}
	for almost every $v\in\nor(K,x)$.
\end{definition}

\begin{prop}\label{prop:subanalytic}
	Let $M$ be a nonempty compact subanalytic set in $\er^d$. Then $M\in\N(\er^d)$.
\end{prop}
\begin{proof}
	This is a well known fact, see e.g. the proof of \cite[Lemma~3.1]{N11}
\end{proof}
\color{black}
\begin{prop}\label{prop:innerCurvature}
	Let $M\subset\er^d$ be a compact set belonging to $\MB_d$, then $M$ satisfies condition (I).
\end{prop}\color{black}
\begin{proof}
	Follows immediately from \cite[Lemma~3]{RZ05}.
\end{proof}

\subsubsection{Curvature measures via normal cycles}
The existence of the normal cycle of a set $M$ allows us to define the so-called curvature measures of $M$ as follows:
given $k\in\{ 0,\ldots,d-1\}$, let $\varphi_k$ be the $k$-th Lipschitz-Killing differential $(d-1)$-form on $\er^{2d}$ which can be described by
\begin{eqnarray*}
	\lefteqn{\langle a^1\wedge\cdots\wedge a^{d-1},\varphi_k(x,n)\rangle}\\
	&=&\cO_{d-k-1}^{-1}\sum_{\sum_i\sigma(i)=d-1-k}\langle\pi_{\sigma(1)}a^1\wedge\cdots\wedge\pi_{\sigma(d-1)}a^{d-1}\wedge n,\Omega_d\rangle,
\end{eqnarray*}
where $a^i$ are vectors from $\er^{2d}$, $\pi_0(x,n)=x$ and $\pi_1(x,n)=n$ are coordinate projections, the sum is taken over finite sequences $\sigma$ of values from $\{0,1\}$, $\Omega_d$ denotes the volume form in $\er^d$ and $\cO_{d-1}=\H^{d-1}(\sphd)=2\pi^{d/2}/\Gamma(\frac d2)$. 

The $k$-th curvature measure of $A$,  $C_k(A,\cdot)$, $k=0,\dots, d-1$, is then defined by
$$C_k(A,F)=(N_A\llc(F\times\er^d))(\varphi_k),\quad F\subset\er^d.$$
We also define the $k$th (total) curvature of $A$ by 
$N_A(\varphi_k)=C_k(A)$ and, finally, the variational measure of $C_k(A,\cdot)$ is denoted by $C_k^{var}(A,\cdot)$. 
Finally, for completeness we also define $C_d(A,F)=\H^d(A\cap F)$. \color{black}


%

\subsection{Delta convex functions and (locally) WDC sets}

 \label{S_WDC}\color{black}

A real function $f$ defined on an open convex set is called DC (delta-convex) when it can be expressed as a difference of two convex functions.
A function $f$ defined on an open set $U$ is said to be locally DC, if for every $x\in U$ there is an open convex set $V\subset U$ containing $s$ such that $f|_{V}$ is DC.
Note that every DC function is locally Lipschitz and that every semi-convex (or semi-concave) function is also DC.
A mapping $F$ (to $\er^d$) defined on a convex set is called a {\it DC mapping} if every component of $F$ is a DC function.

A function $f$ (or a mapping $F$) defined on an interval $[a,b]$ is called DCR if there is a DC function $g$ (DC mapping $G$) defined on some interval $(c,d)\supset [a,b]$ such that $f=g|_{[a,b]}$ ($F=G|_{[a,b]}$).

It is well known that for two DC functions $f,g$ all the functions $f+g$, $fg$, $\max(f,g)$ and $\min(f,g)$ are DC.
Also, if $F,G$ are two DC mappings and $F\circ G$ makes sense then $F\circ G$ is a DC mapping as well (see \cite{VZ89} and \cite{H59}). Apart from those basic properties, we will also need the following results about DC functions.

\begin{lem}{\cite[Lemma 4.8.]{VZ89}}\label{L:Mixing}
	Suppose that $K\subset\er^d$ is open convex and let $f_i: K \to \er$, $i=1,\dots,m$, be DC functions and $f: K \to Y$ a continuous function
	such that $f(x) \in \{f_1(x),\dots,f_m(x)\}$ for each $x \in K$. Then $f$ is DC on $K$.
\end{lem}

Let $f$ be a Lipschitz function on an open set $U\subset\er^d$. A real number $c$ is called a {\it weakly regular value} of $f$ if whenever $x_i \to x$ as $i\to\infty$, with $f(x_i )> f(x) = c$ and $\xi_i \in \partial f(x_i)$, $i\in\en$, then $\xi_i\not \to 0$.
This is equivalent to the condition that for every $K\subset U$ compact there is an $\eps>0$ such that  the inequality $|v|\geq\eps$ holds for every $x\in K$ satisfying $c<f(x)<c+\eps$ and $v\in\partial f(x)$.

A compact set $A\subset\er^d$ is called {\it WDC} ({\it weakly delta-convex}) if there is a DC function $f:\er^d\to [0,\infty)$ with a weakly regular value $0$ such that $A=f^{-1}(\{0\})$.

Note that any compact set $A\subset\er^d$ with positive reach (see \cite{Fe59}) is $\WDC$. Indeed, the distance function $d_A(x)=\dist(x,A)$ is locally semiconvex on an open neighbourhood of $A$ (cf.\   \cite[Satz~(2.8)]{K81}) and locally semiconcave on $\er^d\setminus A$ and therefore DC, and $0$ is a weakly regular value of $d_A$ since $d_A$ has unit gradient at all points $x$ with $0<d_A(x)<\reach(A)$, see \cite[Theorem~4.8]{Fe59}. In particular, compact convex sets are $\WDC$.

We call a set $A\subset\er^d$ {\it locally WDC} if for every $x\in A$ there is $U_x$, an open neighborhood of $x$, and a $\WDC$ set $A_x$ such that $A\cap U_x=A_x\cap U_x.$

The following results about $\WDC$ were proven in \cite{PR13} and \cite{FPR15}.

\begin{thm}{\cite[Theorem~1.2]{PR13}}\label{T:pokornyRatajLocWDC}
	Any compact $\WDC$ set in $\er^d$ admits the normal cycle.
\end{thm}

Note that there are sets in $\er^d$ admitting the normal cycle that are not $\WDC$. A simple example of such set is $B((-1,0),1)\cup B((1,0),1)\subset\er^2$, which is clearly a $\UPR$ set, but cannot be $\WDC$ by \cite[Lemma~7.8]{PRZ18}.\color{black}

%
%
%
The normal cycles of the $\WDC$ sets also have the additivity property, i.e. 
\begin{equation}
N_A+N_B=N_{A\cap B}+N_{A\cup B}
\end{equation}
whenever $A$, $B$, $A\cup B$ and $A\cap B$ are all $\WDC$ (cf. \cite[Theorem~8.8]{PR13}).

The corresponding curvature measures also satisfy other classical formulas of the integral geometry such as the Crofton formula (\cite[Theorem~1.3]{PR13}) and the kinematic formula (\cite[Theorem~B]{FPR15}.
In this paper we will only need the following special case of the kinematic formula which one obtains by applying \cite[(1.2)]{FPR15} for $(M,G)=(\er^d,\overline{SO_d})$, $\beta_i$ being the Lipschitz-Killing forms $\varphi_i$ and $\phi$ and $\psi$ being the characteristic functions of $U$ and $V$, respectively:

\begin{thm}[Kinematic formula for $\WDC$ sets]\label{thm:kinematicForWDC}
	Let $A$ and $B$ be two compact $\WDC$ sets in $\er^d$ and let $0\leq k\leq d-1$.
	Then $A\cap g(B)\in\UWDC$ for almost every $g\in\G_d$ and
	\begin{equation}\label{eq:kinematicFormulaFroWDC}
	\int_{{\cal G}_d}C_k(A\cap g (B),U\cap g (V))\, dg=\sum_{i+j=d+k}\gamma_{d,i,j}C_i(A,U)C_j(B,V),
	\end{equation}
	where $\gamma_{d,i,j}$ are constants depending only on $d$, $i$ and $j$.
\end{thm}
\color{black}

\color{black}
\begin{prop}\label{P:intersectingBalls}
	Let $\cM$ and $\K$ be finite nonempty collections of $\WDC$ sets in $\er^d$. Then $\{M\cap g(K):M\in\cM,\; K\in\K \}$ is again a collection of $\WDC$ sets for almost every $g\in\G_d$. 
	
	In particular, if $r>0$ then $\{M\cap B(x,r),M\in\cM\}$ a collection of $\WDC$ sets for almost every $x\in\er^d$.
\end{prop}

\begin{proof}
  Pick $M\in\cM$ and $K\in\K$.
  By \cite[Proposition 4.1]{FPR15} we know that $M\cap g(K)$ is $\WDC$ for almost every $g\in \G_d$. Since both $\cM$ and $\K$ are finite (and so in particular countable), $M\cap g(K)$ is $\WDC$ for almost every euclidean motion $g$ on $\er^d$ and every $M\in\cM$ and $K\in\K$.
  
  The second part of the lemma 
  follows directly form de construction of the Haar measure on $\G_d$.
\end{proof}
\color{black}
\color{black}
\subsection{Lipschitz and DC graphs}\label{subs:LipADCgRAPHS}
	A set $M\subset\er^d$ is called a Lipschitz (DC) graph in the direction $v\in\sph^{d-1}$ if there is a closed convex set $K\subset v^{\bot}$
	and a Lipschitz (DCR) function $f:K\to \spa (v)$ such that
	\begin{equation*}
	M=\{t+f(t):t\in K\}.
	\end{equation*}
	We will say that $M$ is an $L$-Lipschitz graph in the direction $v$, if the function $f$ in the definition above can be found $L$-Lipschitz.
	We also say that $M$ is a Lipschitz ($L$-Lipschitz) graph, if it is an 
	$L$-Lipschitz graph in the direction $v$ for some $v\in\sph^{d-1}$.
	
	The following results will be useful.
	
	\begin{lem}{\cite[Lemma~7.3]{PRZ18}}\label{L:zestr}
		Let $P$ be a DC graph in $\R^2$ and $0 \in P$.
		Suppose that $\Tan(P,0)$ is a $1$-dimensional space and $(0,1) \notin \Tan(P,0)$. Then there exists 
		$\rho^*>0$ such that, for each $0< \rho< \rho^*$, there exist $\alpha<0 < \beta$ and a DCR function $f$ on $(\alpha, \beta)$ such that
		$P \cap B(0, \rho) = \graph  f|_{(\alpha, \beta)}$.
	\end{lem}
	
	\begin{lem}\label{L:LipschitzGraphEquivalence}
		Suppose that $M\subset\er^d$, $v\in\sph^{d-1}$ and $L>0$.
		Then the following conditions are equivalent:
		\begin{enumerate}[label={ (\alph*)}]
			\item\label{cond:LipschitzGraph} $M$ is an $L$-Lipschitz graph in the direction $v$,
			\item\label{cond:ProjectionsAndProducts} $\Pi_{v}(M)$ is convex and
			\begin{equation*}\label{eq:inequalityInConditionB}
			\left|(A-B)\cdot v\right|\leq \frac{L}{\sqrt{1+L^{2}}}\;|A-B|
			\end{equation*}
			for every $A,B\in M$.
		\end{enumerate}
	\end{lem}
\begin{proof}
	Suppose that $v\in \sph^{d-1}$, $K\subset v^{\bot}$ and $f:K\to\spa(v)$.
	Then
	\begin{equation}\label{eq:seriesOfEquivalencies}
	\sqrt{1+L^{2}}\;\left|(A-B)\cdot v\right|\leq L|A-B|\iff |f(s)-f(t)|\leq L|s-t|,
	\end{equation}
	whenever $A=s+f(s)$ and $B=t+f(t)$, $s,t\in K$.
	Indeed, the definition of $f$ implies
	\begin{equation*}
	|(A-B)\cdot v|=|[s-t+f(s)-f(t)]\cdot v|=|(f(s)-f(t))\cdot v|=|f(s)-f(t)|
	\end{equation*}
	and 
	\begin{equation*}
	|A-B|^2=|f(s)-f(t)|^2+|s-t|^2.
	\end{equation*}
	So we can write the following series of equivalences.
	\begin{equation}
	\begin{aligned}
	|f(s)-f(t)|\leq L|s-t|
	&\iff |f(s)-f(t)|^2\leq L^2|s-t|^2\\
	&\iff |f(s)-f(t)|^2\leq L^2(|A-B|^2-|f(s)-f(t)|^2)\\
	&\iff (L^2+1)|f(s)-f(t)|^2\leq  L^2|A-B|^2\\
	&\iff (L^2+1)|(A-B)\cdot v|^2\leq L^2|A-B|^2\\
	&\iff \sqrt{L^2+1}|(A-B)\cdot v|\leq L|A-B|.
	\end{aligned}
	\end{equation}
	Which is what we want.
	
	Now, to prove the implication 
	$\ref{cond:LipschitzGraph}\implies\ref{cond:ProjectionsAndProducts}$,
	suppose that $M$ is an $L$-Lipschitz graph with corresponding $v\in\sph^{d-1}$, $K\subset v^{\bot}$ and $f:K\to\spa(v)$.
	Clearly $\Pi_{v}(M)$ is convex since $\Pi_{v}(M)=K$ and $K$ is convex.
	Moreover, the $L$-Lipschitzness of $f$ and 
	\eqref{eq:seriesOfEquivalencies} imply \ref{eq:inequalityInConditionB}, which concludes the proof on the implication.
	
	To prove the opposite implication assume \ref{cond:ProjectionsAndProducts} 
	and put $K=\Pi_{v}(M)$.
	By \ref{eq:inequalityInConditionB} we, in particular, know that $\Pi_v$ is injective on $M$.
	Indeed, it $\Pi_v$ was not injective on $M$, there would be $A,B\in M$ and $\alpha\in\er$ such that $B=A+\alpha v$.
	Then \ref{eq:inequalityInConditionB} implies
	\begin{equation*}
	|\alpha|= |(A-B)\cdot v|\leq \frac{L}{\sqrt{1+L^{2}}}\;|A-B| <|A-B|=|\alpha|,
	\end{equation*}
	which in not possible.
	
	Therefore we can define $f:K\to\spa(v)$ by $f(s)=\Pi_{v}^{-1}(s)\cap M-s$.
	Now, if $A,B\in M$ then there are $s,t\in K$ such that $A=s+f(s)$ and $B=t+f(t)$ and so we can again use \eqref{eq:seriesOfEquivalencies} to obtain that $f$ is $L$-Lipschitz and so $M$ is a Lipschitz graph.
\end{proof}
\color{black}
\begin{rem}\label{rem:DCgraphProperties}
Suppose that $P$ is a DC graph in $\er^2$ with $K=v^\bot$. Then the following is true (see \cite[Remark~7.1]{PRZ18} for the proof, note that Lipschitzness of $f$ is not needed in the proof): for $a=c + \vf(c) \in P$ 	there exist DC graphs $P_1, P_2 \subset \R^2$ such that $P \subset P_1 \cup P_2$, $a \in P_1\cap P_2 $
and $\Tan(P_i,a)$ is a $1$-dimensional space, $i=1,2$.
\end{rem}

\subsubsection{Lipschitz domains}
\color{black}

Similarly to the definition of the Lipschitz graph we say that $\varnothing \neq A \subset \R^d$ is a Lipschitz domain if it can be locally represented as subgraph of a Lipschitz function (cf. \cite[Section~3.1]{PR13}, or \cite{RZ03}, where the term $d$-dimensional Lipschitz manifold was used).

We will need the following easy observation about Lipschitz domains:

\begin{lem}\label{lem:LipsechitzParts}
	Let $M$ be a nonempty compact Lipschitz domain in $\er^2$ then
	\begin{enumerate}[label={ (\alph*)}]
		\item\label{cond:componentsJordan} each connected component of $\partial M$ is a Jordan curve,
		\item\label{cond:JordanPartition} if $\gamma:[a,b]\to\er^2$ is a Jordan curve as in \ref{cond:componentsJordan} then
		there is a partition $\D\coloneqq\{a=t_0<\cdots <t_n=b\}$ such that the image of
		$\gamma$ restricted to $[t_i,t_{i+1}]$ is a Lipschitz graph for every $i=0,\dots,N-1$. 
	\end{enumerate}
\end{lem}

\begin{proof}
	Part~\ref{cond:componentsJordan} follows from \cite[Theorem~6.1]{L00} and part~\ref{cond:JordanPartition} directly from the compactness of $M$ (and the definition of a Lipschitz domain).
\end{proof}

\color{black}
\subsection{Curves of finite turn}
A curve $\gamma:[a,b]\to\er^d$ is said to have a finite turn if there is a constant $K\in\er$ such that
\begin{equation*}
\begin{aligned}
\sum_{k=1}^{n-1}&\rho\left(\frac{\gamma_{i}(x_{k})-\gamma_{i}(x_{k-1})}
{|\gamma_{i}(x_{k})-\gamma_{i}(x_{k-1})|},
\frac{\gamma_{i}(x_{k+1})-\gamma_{i}(x_{k})}
{|\gamma_{i}(x_{k+1})-\gamma_{i}(x_{k})|}\right)
<K
\end{aligned}
\end{equation*}
for every partition $a=x_0<\cdots<x_n=b$ of $[a,b]$ such that the sum on the left hand side makes sense.
The above definition can be equivalently formulated that there is a constant $K\in\er$ such that
\begin{equation*}
\begin{aligned}
\sum_{k=1}^{n-1}&\left|\frac{\gamma_{i}(x_{k})-\gamma_{i}(x_{k-1})}
{|\gamma_{i}(x_{k})-\gamma_{i}(x_{k-1})|}-
\frac{\gamma_{i}(x_{k+1})-\gamma_{i}(x_{k})}
{|\gamma_{i}(x_{k+1})-\gamma_{i}(x_{k})|}\right|
<K
\end{aligned}
\end{equation*}
for every partition $a=x_0<\cdots<x_n=b$ such that the sum on the left hand side makes sense. Note that  in \cite{D08} the latter is used as a definition of a curve with a finite turn, whereas our definition of finite turn is (in \cite{D08}) referred to as a curve of a finite angular turn.
We will use the following results:

%

\color{black}
\begin{lem}\label{L:finiteTurnIsDC}
	Let $\gamma:[a,b]\to\er$ be a curve of finite turn and suppose that $\Im(\gamma)$ is a Lipschitz graph. Then $\Im(\gamma)$ is a DC graph.
\end{lem}

\begin{proof}
	Without any loss of generality we may assume that $\gamma(t)=(t,f(t))$, $t\in[a,b]$, for some Lipschitz function $f:[a,b]\to\er$.
	
	Pick $T\in[a,b]$ and define $\varphi:[a,b]\to\er$ by the formula $\varphi(t)\coloneqq\int_{T}^{t}\sqrt{1+(f'(s))^2}\;ds$.
	Put $[c,d]\coloneqq\varphi[a,b]$
	Then $\gamma\circ\varphi^{-1}$ is the arc length parametrisation of $\gamma$.
	Moreover, since for $a<u<t<b$
	\begin{equation*}
	|\varphi(t)-\varphi(u)|=\left|\int_{T}^{t}\sqrt{1+(f'(s))^2}\;ds
	-\int_{T}^{u}\sqrt{1+(f'(s))^2}\;ds\right|
	=\left|\int_{u}^{t}\sqrt{1+(f'(s))^2}\;ds\right|
	\end{equation*}
	and 
	\begin{equation*}
	|t-u|\leq\left|\int_{u}^{t}\sqrt{1+(f'(s))^2}\;ds\right|\leq |t-u|\sqrt{1+L^2},
	\end{equation*}
	we have that both $\varphi$ and $\varphi^{-1}$ are biLipschitz.
	
	By \cite[Proposition~5.7~(i)]{D08} we know that $\gamma\circ\varphi^{-1}$ is
	(as an arc length parametrization of a curve with a finite turn) DCR and so both $\varphi^{-1}$ and $f\circ\varphi^{-1}$ are (as its coordinates) also DCR.
	Moreover, \cite[Remark~5.6~(i)]{D08} and the fact that $\varphi^{-1}$ is biLipschitz imply that $\varphi$ is also DCR. 
	So (by \cite[Remark~5.6~(ii)]{D08}) $f=f\circ\varphi^{-1}\circ\varphi$ is DC,
	which concludes the proof.
\end{proof}\color{black}

\section{$\UWDC$ sets ant their normal cycles}\label{sec:curvatures}
\color{black}
	
\begin{definition}\label{def:UWDC}
A set $M\subset\er^d$ belongs to $\UWDC(\er^d)$ (or is a $\UWDC$ set) if for every $x\in M$ there is a neighbourhood $U$ of $x$ and sets $M_1,\dots,M_{j}$ such that
$M\cap U=U\cap \bigcup_{i=1}^{j} M_i$ and such that each set $M_{I}:=\bigcap_{i\in I} M_i$, $I\in\Sigma_j$, is $\WDC$.

We will also define an auxiliary class $\UWDCG$ of the sets $M=\bigcup_{i=1}^{N} M_i\subset\er^d$ such that each set $M_{I}:=\bigcap_{i\in I} M_i$, $I\in\Sigma_N$, is compact $\WDC$.	
\end{definition}

\begin{rem}\label{rem:UWDSremarks}
	\begin{enumerate}[label={ (\alph*)}]
		\item We will sometimes omit the argument $\er^d$ and just write $M\in\UWDC$ ($M\in\UWDCG$) if the dimension $d$ is already specified. We will also often write $M=\bigcup_{i=1}^{j} M_i\in\UWDCG$ implying that the sets $M_i$ are the corresponding sets in the definition of a $\UWDCG$ set (related to $M$).
		\item Due to the local nature of the notion of $\UWDC$ sets, it makes no difference if the assumption that $M_I$ is $\WDC$ in the definition of $\UWDC$ sets is replaced with the assumption that $M_I$ is locally $\WDC$. 
		\item Note that it follows directly from the definition above that each $\UWDC$ set is immediately also a $\UPR$ set (since the definition of $\UPR$ sets is essentially the same except the sets $M_I$ are assumed to have a positive reach which means that they are locally $\WDC$ as well)
		\item\label{item:remarkIntersection} It is easy to see that (by Proposition~\ref{P:intersectingBalls}) a set $M\subset\er^d$ is $\UWDC$ if and only if for every $x\in M$ there is a closed ball $B$ containing $x$ in the interior (but not necessarily centred at $x$) such that $M\cap B$ is $\UWDCG$.
	\end{enumerate}
\end{rem}\color{black}

\color{black}
\begin{lem}
Every $M\in\UWDC$ is locally contractible.
\end{lem}

\begin{proof}
	This is essentially the proof of \cite[Proposition~2.1]{RZ01}.
Pick $x\in M$ and suppose that $M_i$, $i=1,\dots,j$, are as in the definition of $\UWDC$ set for some neighbourhood $U$ of $x$.
Due to Proposition~\ref{P:intersectingBalls} we can assume that $x\in \bigcap_{i=1}^{j} M_i$.
By \cite[Lemma~3.1]{PRZ18} there is an $\eps>0$ together with a (strong) deformation retractions $\Phi_I: G_I\times[0,1]\to M_I:=\bigcap_{i\in I} M_i$, $I\in\Sigma_j$, with $(M_I)_{\eps}\subset G_I$ and $G_I$ open.
Those induce projections $P_I:G_I\to M_I$ defined by $P_I(z)=\Phi_I(z,1)$.
Using those projections we will define a projection $P:B(x,\eps)\to M$ as follows.

For $z\in B(x,\eps)$ denote $d_i(z):=\dist(z,M_i)$, $i=1,\dots,j$.
For $\sigma$, a permutation on $\{1,\dots,j\}$, define 
\begin{equation*}
M_\sigma\coloneqq\left\{ z\in B(x,\eps): d_{\sigma(i)}(z)\leq d_{\sigma(i+1)}(z),\; i=1,\dots,j-1 \right\}.
\end{equation*}
Denote $I_i^\sigma=\{\sigma(1),\dots,\sigma(i)\}$, $i=1,\dots,j$.
For $\sigma$, a permutation on $\{1,\dots,j\}$, and $z\in M_\sigma $ put $z^\sigma_{j}=P_{I^\sigma_{j}}(z)$ and define
\begin{equation*}
z^\sigma_i=P_{I^\sigma_{i}}\left(\left(1-\frac{d_i(z)}{d_{i+1}(z)}\right)P_{I^\sigma_{i}}(z)
+\frac{d_i(z)}{d_{i+1(z)}}z^\sigma_{i+1}\right),\quad i=1,\dots j-1.
\end{equation*}
Note that if $d_{\sigma(i)}(z)= d_{\sigma(i+1)}(z)$ then $z^\sigma_{i}=z^\sigma_{i+1}$ and so $z^\sigma_{i}=z^\tau_{i}$, $i=1,\dots,j$, whenever $z\in M_\sigma\cap M_\tau$.
Therefore the mapping $P(z)\coloneqq z^\sigma_1$, $z\in M_\sigma$ is well defined.
The mapping $P$ is also continuous on each $M_\sigma$ due to the continuity of all mappings $P_I$ and so $P$ is continuous on $B(x,\eps)$.
To conclude the proof it remains to define a contraction $\Xi:(M\cap B(x,\eps))\times [0,1]\mapsto \{x\}$.
This can be done in a standard way by
$
\Xi(z,t)=P((1-t)z+tx).
$
\end{proof}

\begin{cor}
Let $M,K,M\cap K\in \UWDC$ then $\chi(M)$, $\chi(K)$, $\chi(M\cap K)$ and $\chi(M\cup K)$ are all well defined and
\begin{equation*}
\chi(M)+\chi(K)=\chi(M\cap K)+\chi(M\cup K).
\end{equation*}
\end{cor}

\begin{proof}
The proof is exactly the same as the one of \cite[Proposition~2.2]{RZ01}.
\end{proof}\color{black}

\color{black}
Consider $M=\bigcup_{i=1}^{N} M_i\in\UWDCG$. Then $N_{M_{I}}$, the normal cycle of $M_I\coloneqq\bigcap_{i\in I}M_i\in\WDC$, exists (by Theorem~\ref{T:pokornyRatajLocWDC}) for every $I\in\Sigma_N$ and we can define an integral current $T(M)$ as 
\begin{equation}
T(M):=\sum\limits_{n=1}^{N}(-1)^{n+1}\sum\limits_{|I|=n} N_{M_I}.
\end{equation}
Clearly, $T(M)$ is a Legendrian cycle.
Let $S:=\bigcup_{I}\supp T_{M_I}$.
Then $\supp T(M)\subset S$ and therefore if a half space $H$ does touch neither of $N_{M_I}$ it does not touch $T(M)$ either.
Denote the system of all half spaces $H$ that touch neither $N_{M_I}$ by $\H$.
Since the currents $ N_{M_I}$ are normal cycles $\H$ is of full measure and so $T(M)$ satisfies \eqref{eq:NormalCycleTouching}.

Moreover, for every $I\in\Sigma_N$ there is a set $\H_I\subset\sphd\times\er$ of a full $d$-dimensional Hausdorff measure such that $M_I\cap H_{v,t}$ is $\WDC$ and that
\begin{equation}\label{eq:NormalCycleEulerLoc}
\langle T_{M},\pi_1,-v\rangle(H_{v,t}\times \sphd)=\chi(M\cap H_{v,t}).
\end{equation} 
holds for every $(v,x)\in\H_I$.


To prove that $T(M)$ is the normal cycle of $M$, it is enough to show that \eqref{eq:NormalCycleEuler} holds for every $(x,v)\in\bigcap_{I\in\Sigma_N}\H_I\eqqcolon \widetilde \H$.
To do that pick some $(x,v)\in\widetilde \H$ and we want to prove that \eqref{eq:NormalCycleEulerLoc} holds.

First note that the right hand side of \eqref{eq:NormalCycleEulerLoc} makes sense since
\begin{equation*}
\widetilde M\coloneqq M\cap H_{v,t}=\bigcup_{i=1}^{N} (M_i\cap H_{v,t})\eqqcolon \bigcup_{i=1}^{N}\widetilde M_i
\end{equation*}
is $\UWDC$ by the that each of the sets $\widetilde M_I$ defined by
\begin{equation*}
\widetilde M_I=\bigcap_{i\in I} \widetilde M_i=H_{v,t}\cap\bigcap_{i\in I}  M_i
\end{equation*}
is $\WDC$.
Now,
\begin{equation*}
\begin{aligned}
\langle T(M),\pi_1,-v\rangle(H_{v,t}\times \sph^{d-1})&=\left\langle\sum\limits_{n=1}^{N}(-1)^{n+1}\sum\limits_{|I|=n} N_{M_I},\pi_1,-v\right\rangle(H_{v,t}\times \sph^{d-1})\\
&=\sum\limits_{n=1}^{N}(-1)^{n+1}\sum\limits_{|I|=n}\left\langle N_{M_I},\pi_1,-v\right\rangle(H_{v,t}\times \sph^{d-1})\\
&=\sum\limits_{n=1}^{N}(-1)^{n+1}\sum\limits_{|I|=n}\chi(M_I\cap H_{v,t})\\
&=\chi(M\cap H_{v,t}).
\end{aligned}
\end{equation*}

\color{black}
Thus we obtain the following:
\begin{lem}\label{l:globalNormalCycle}
Every compact $M=\bigcup_{i=1}^{N} M_i\in\UWDCG$ admits the normal cycle $N_M$.
Moreover, $N_M$ is satisfies the formula
\begin{equation}\label{E:N_M via additivity}
N_M=\sum\limits_{n=1}^{N}(-1)^{n+1}\sum\limits_{|I|=n} N_{M_I}\,
\end{equation}
\end{lem}

This allows us to prove our first main result.
\color{black}
\begin{thm}\label{thm:UWDChasNC}
		Each compact $\UWDC$ set admits the normal cycle.
\end{thm}

\begin{proof}
Pick a compact set $M\in\UWDC(\er^d)$. We may assume that $M\not=\varnothing$.
	By the compactness of $M$ and the definition of a $\UWDC$ set there are 
	$n\in\en$,
	$x_1,\dots,x_n\in M$, $r_1,\dots,r_n>0$, $i_1,\dots, i_n\in\en$ and $M_i^j\subset\er^d$, $j=1,\dots,n$, $i=1,\dots, i_j$ such that:
	\begin{enumerate}[label={ (\alph*)}]
		\item\label{cond:sameIntersections} $M\cap B(x_j,3r_j)=\left(\bigcup_{i=1}^{i_j} M_i^j\right)\cap B(x_j,3r_j)$ for every $j=1,\dots,n$,
		\item\label{cond:Mcovered} $M\subset\bigcup_{j=1}^n B(x_j,r_j)$,
		\item\label{cond:intersectionsWDC} $M_I^j\coloneqq\bigcap_{i\in I} M^j_i$ is $\WDC$ for every $I\in\Sigma_{i_j}$ and every $j=1,\dots,n$.
	\end{enumerate}
	
Next we claim that we can find balls $B_k=B(y_k,2r_k)$, $k=1,\dots,n$ such that for every for every $k=0,\dots,n$ the following conditions hold:
\begin{enumerate}[label={ (\Alph*)}]
	\item\label{cond:ballInside} if $k>0$ then $
	B(x_k,r_k)\subset B(y_k,2r_k)\subset B(x_k,3r_k)
	$
	\item \label{cond:WDCsystem} the system
	\begin{equation}
	\sM_k\coloneqq 
	\begin{cases}
	\{B_J\cap M_I^j:\; J\in\Sigma_k, I\in\Sigma_{i_j} \}\cup \{M_I^j:\;  I\in\Sigma_{i_j} \}&\text{if}\quad k>0\\ 
	\{M_I^j:\;  I\in\Sigma_{i_j} \}&\text{if}\quad k=0
	\end{cases}
	\end{equation}
	is a system of $\WDC$ sets, where we denote $B_J=\bigcap_{j\in J} B_j$, $J\in\Sigma_k$.
\end{enumerate}
This will be done by induction.

For $k=0$ we just need to verify condition~\ref{cond:WDCsystem}, but that follows directly from condition~\ref{cond:intersectionsWDC}.
For the induction step assume that we have balls $B_1,\dots,B_{k}$ constructed with both conditions \ref{cond:ballInside} and \ref{cond:WDCsystem} satisfied for some $k\in\{0,\dots,n-1\}$.

Using the second part of Proposition~\ref{P:intersectingBalls} with $r=2r_{k+1}$ and $\sM=\sM_{k}$ (which is possible by the induction procedure) we can find $B_{k+1}$ such that \ref{cond:ballInside} holds and such that
\begin{equation*}
\widetilde\sM\eqqcolon \{ B_k\cap K: K\in\sM_{k} \}
\end{equation*}
is a system of $\WDC$ sets. 
Now it is enough to observe that $\sM_{k+1}=\widetilde{\sM}\cup \sM_{k}$ which proves the claim.

For $J\in\Sigma_n$ define 
\begin{equation*}
L_J\coloneqq M\cap B_J.
\end{equation*}
Then 
\begin{equation*}
L_J=\bigcup_{i=1}^{i_j} B_J\cap M_i^j
\end{equation*}
whenever $j\in J$ by \ref{cond:sameIntersections}.
Since $B_J\cap M_I^j\in\sM_n$, $I\in\Sigma_{i_j}$, we obtain (by \ref{cond:WDCsystem}) that $L_J$ is $\UWDCG$, $J\in \Sigma_n$.

Moreover, 
\begin{equation*}
M=\bigcup_{|I|=1} L_I
\end{equation*}
by \ref{cond:Mcovered} and \ref{cond:ballInside}.
Therefore we can again define an integral current
\begin{equation}
T(M):=\sum\limits_{n=1}^{N}(-1)^{n+1}\sum\limits_{|I|=n} N_{L_I}.
\end{equation}
	The fact that $T(M)$ is indeed the normal cycle of $M$ follows the same way as in the proof of Lemma~\ref{l:globalNormalCycle}. 
\end{proof}\color{black}
\color{black}
\begin{thm}[Kinematic formula]\label{thm:kinematicFormula}
	Let $M$ and $K$ be two compact $\UWDC$ sets in $\er^d$  and let $0\leq k\leq d-1$.
	Then $M\cap g(K)\in\UWDC$ for almost every $g\in\G_d$ and
	\begin{equation}\label{eq:KF}
	\int_{{\cal G}_d}C_k(M\cap g K,U\cap g V)\, dg=\sum_{i+j=d+k}\gamma_{d,i,j}C_i(M,U)C_j(K,V),
	\end{equation}
	where $\gamma_{d,i,j}$ are constants depending only on $d$, $i$ and $j$.
\end{thm}

\begin{proof}
We first prove the case $M=\bigcup_{i=1}^{p} M_i\in\UWDCG(\er^d)$ and 
	$K=\bigcup_{j=1}^{q} K_j\in\UWDCG(\er^d)$.
	Denote as usual $M_I\coloneqq \bigcap_{i\in I} M_i$, $I\in\Sigma_p$, and 
	$K_J\coloneqq \bigcap_{j\in J} K_j$, $J\in\Sigma_q$.
	First note that the set $g(K)$ is again $\UWDCG$ for every Euclidean motion $g$.
	This is because 
	$g(K)=\bigcup_{j=1}^{q} g(K_j)$ and 
	$\bigcap_{j\in J} g(K_j)=g(K_J)\in\WDC$, $J\in\Sigma_q$.
	
Next we prove that $M\cap g(K), M\cup g(K)\in\UWDCG$ for almost every $g\in\G_d$.
Proposition~\ref{P:intersectingBalls} applied for $\cM\coloneqq \{M_I:I\in\Sigma_p\}$ and $\K\coloneqq \{K_J:J\in\Sigma_q\}$ implies that there is $G_0\subset\G_d$ of full measure such that $M_I\cap g(K_J)\in\WDC$ whenever $g\in G_0$, $I\in\Sigma_p$, $J\in\Sigma_q$.
Since $M\cap g(K)=\bigcup_{i,j}M_i\cap g(K_j)$ one can easily see that $M\cap g(K)\in\UWDCG$, $g\in G_0$, since every intersection of a nonempty and finite collection of the sets $M_i\cap g(K_j)$ is of the form $K_I\cap g(K_J)\in\WDC$.
Similarly, $M\cup g(K)\in\UWDCG$ since
$M\cup g(K)=\left(\bigcup_{i}M_i\right)\cup \left(\bigcup_{j}g(K_j)\right)$
and again every intersection of a nonempty and finite collection of the sets in the union is of the form $M_I$, $g(K_J)$ or $M_I\cap g(K_J)$, $I\in\Sigma_p$, $J\in\Sigma_q$, which are all $\WDC$ sets.
Note that this can be also expressed in a more convenient way (which we will also use later) by $M_I\cap g(K_J)\in\WDC$, $I\in\Sigma^0_p$, $J\in\Sigma^0_q$, $|I|+|J|\geq 1$, where we put $M_\varnothing=K_\varnothing\coloneqq \er^d$.

	Denote
	\begin{equation*}
	L\coloneqq \int_{{\cal G}_d}C_k(M\cap g( K),U\cap g (V))\, dg
	\quad\text{and}\quad
	P\coloneqq \sum_{i+j=d+k}\gamma_{d,i,j}C_i(M,U)C_j(K,V).
	\end{equation*}
We want to prove $L=P$.
	Applying $\varphi_k$ to both sides of \eqref{E:N_M via additivity} we obtain
	\begin{equation*}
	C_k(M,\cdot)=\sum\limits_{n=1}^{p}(-1)^{n+1}\sum\limits_{|I|=n} C_k(M_I,\cdot)=\sum_{\substack{I\in \Sigma_p}}
	(-1)^{|I|+1}C_k(M_I,\cdot)
	\end{equation*}
	and
	\begin{equation*}
	C_k(K,\cdot)=\sum_{\substack{J\in \Sigma_q}}.
	(-1)^{|J|+1}C_k(K_J,\cdot),
	\end{equation*}
	Similarly, one can also see that
	\begin{equation*}
	C_k(M\cup g (K),\cdot)=\sum_{\substack{I\in \Sigma_p^0, J\in\Sigma_q^0 ,\\|I|+ |J|\geq 1}}
	(-1)^{|I|+|J|+1}C_k(M_I\cap g(K_J),\cdot),
	\end{equation*}
	whenever $g\in G_0$.
	Hence,
	\begin{equation*}
	\begin{aligned}
	L&=\int_{{\cal G}_d}C_k(M,U\cap g (V))\, dg
	 +\int_{{\cal G}_d}C_k(g (K),U\cap g (V))\, dg\\
	 &-\int_{{\cal G}_d}C_k(M\cup g (K),U\cap g( V))\, dg\\
	 &=\int_{{\cal G}_d}
	 \sum_{\substack{I\in \Sigma_p}}
	  (-1)^{|I|+1}\;C_k(M_{I},U\cap g( V))
	 +\sum_{\substack{J\in \Sigma_q}}
	 (-1)^{|J|+1}C_k(g (K_{J}),U\cap g (V))\, dg\\
	 &-\int_{{\cal G}_d}C_k(M\cup g (K),U\cap g (V))\, dg.
	\end{aligned}
	\end{equation*}
	Also
	\begin{equation*}
	\begin{aligned}
	P&=\sum_{i+j=d+k}\gamma_{d,i,j}
	\left( \sum_{\substack{I\in \Sigma_p }}
	 (-1)^{|I|+1} C_i(M_I,U)\right)
	\cdot \left( \sum_{\substack{J\in \Sigma_q}}
	 (-1)^{|I|+1} C_j(K_J,V)\right)\\
	 &=\sum_{i+j=d+k}\gamma_{d,i,j}
	\sum_{\substack{I\in \Sigma_p, J\in\Sigma_q }} (-1)^{|I|+|J|}\; C_i(M_I,U)\cdot C_j(K_J,V)\\
	&=\sum_{\substack{I\in \Sigma_p, J\in\Sigma_q}} (-1)^{|I|+|J|}
	\sum_{i+j=d+k}\gamma_{d,i,j}\; C_i(M_I,U)\cdot C_j(K_J,V)\\
	&=\sum_{\substack{I\in \Sigma_p, J\in\Sigma_q }} (-1)^{|I|+|J|}
	\int_{{\cal G}_d}C_k(M_I\cap g(K_J),U\cap g (V))\, dg\\
	&=\int_{{\cal G}_d}\sum_{\substack{I\in \Sigma_p, J\in\Sigma_q }}
	(-1)^{|I|+|J|}C_k(M_I\cap g(K_J),U\cap g (V))\, dg,
	\end{aligned}
	\end{equation*}
	Where the second to last equality holds by Theorem~\ref{thm:kinematicForWDC}.
	Moreover, for $g\in G_0$,
	\begin{equation*}
	\begin{aligned}
	C_k(M\cup g (K),\cdot)&=\sum_{\substack{I\in \Sigma_p^0, J\in\Sigma_q^0 ,\\|I|+ |J|\geq 1}}
	(-1)^{|I|+|J|+1}C_k(M_I\cap g(K_J),\cdot)\\
	&=\sum_{\substack{I\in \Sigma_p, J\in\Sigma_q }}
	(-1)^{|I|+|J|+1}C_k(M_I\cap g(K_J),\cdot)\,\\
	&+\sum_{\substack{I\in \Sigma_p }}
	(-1)^{|I|+1}\;C_k(M_{I},\cdot)
	+\sum_{\substack{J\in \Sigma_q}}
	(-1)^{|J|+1}C_k(g(K_J),\cdot),
	\end{aligned}
	\end{equation*}
	and applying this to $U\cap g(V)$ and integrating over $\G_d$ we obtain that $L=P$ which concludes the proof.
\end{proof}
\color{black}

\section{$\UWDC$ sets in plane}\label{sec:inPlane}

In this section we aim to provide a simple geometric characterization of compact $\WDC$ sets in $\er^2$ (Theorem~\ref{T:characterizationUWDC}). We start with some simple observations and few definitions.

\color{black}
\begin{prop}\label{P:numberOfComponentsBound}
	Let $A\subset\er^d$ and let $B_1,\dots, B_N$ be a finite covering of $A$.
	Suppose that $A\cap B_i$ has finitely many connected components relatively in $A\cap B_i$ for every $i=1,\dots,N$.
	Then $A$ has finitely many connected components.
\end{prop}

\begin{proof}
	Let $C_i^j$ be the connected components of $A\cap B_i$ and denote the system of all $C_i^j$ by $\C$.
	Let $\A$ be the system of all connected components of $A$. For each $D\in\A$ pick some open (relatively in $A$) set $U_D$ such that $D\subset U_D$ and that $U_D\cap U_E=\varnothing$ whenever $D,E\in\A$, $D\not= E$.
	Since every $C\in \C$ is connected and the system $\U\coloneqq\{U_D:D\in\A \}$ is an open covering  (relatively in $A$) of $A$ we know that there is a unique $U^C\in \U$ such that $U^C\cap C\not=\varnothing$.
	Moreover, for each $U\in\U$ there is at least one $C\in\C$ such that $C\cap U\not=\varnothing$ (this is because the system $\C$ is a covering of $A$) and so the mapping $C\mapsto U^C$ maps $\C$ onto $\U$.
	Hence, $|\A|=|\U|\leq|\C|<\infty$.
\end{proof}

We will also use the following easy fact which we state without a proof.

\begin{prop}\label{P:numberOfComponentsBound2}
	Let $A$ be a subset of a metric space $M$ such that $A\cap C\not=\varnothing$ for every $C$ a connected component of $M$.
	If $A$ has $N$ connected components, then $M$ has at most $N$ connected components.
\end{prop}
\color{black}

\color{black}
Let us recall some definitions and notation from \cite[Definition~7.4]{PRZ18}.
If $z\in \R^2$ and  $v\in \sph^1$ we denote by   $\gamma_{z,v}$ the unique orientation preserving isometry on $\er^2$ that maps $0$ to $z$ and $(1,0)$ to $z+v$. 
If $K\subset \er$ and $f:K\to\er$ is a function, then $\hyp f$ and $\epi f$ will be used for hypograph and epigraph of $f$, respectively;
\begin{equation*}
\hyp f\coloneqq \{(x,y)\in\er^2: x\in K,\ y\leq f(x)\},\ \ \ \epi f\coloneqq \{(x,y)\in\er^2: x\in K, \ y\geq f(x)\}.
\end{equation*}

Further, for $u>0$, $s\in(0,\infty]$, $z\in\er^2$ and $v\in \sph^1$, we define
\begin{equation*}
A_s^u\coloneqq  \left\{(x,y): 0\leq x< s, -xu\leq y\leq xu \right\}\ \ \text{and}\ \ A_s^u(z,v)\coloneqq \gamma_{z,v}(A_s^u).
\end{equation*}

\begin{definition}{(\cite[Definition~7.4]{PRZ18})} \label{types}
	Let $M \subset \R^2$ and $r, u>0$. We say that
	\begin{enumerate}
		\item[(i)]
		$M$ is a {\it $\tilde T_{r,u}^1$-set} if $M\cap A_{r}^{2u}=\{0\}$.
		\item[(ii)]
		$M$ is a {\it $\tilde T_{r,u}^2$-set} if   $M\supset A_{r}^{2u}$.
		\item[(iii)]
		$M$ is a  {\it $\tilde T_{r,u}^3$-set} if there is a DCR function $U:[0,r)\to\er$ such that $U'_{+}(0)=0$, $\graph U\subset A_{r}^{u}$ and 
		$M\cap A_{r}^{2u}=\hyp U\cap A_{r}^{2u}$.
		\item[(iv)]
		$M$ is a  {\it $\tilde T_{r,u}^4$-set} if  there is a DCR function $L:[0,r)\to\er$ such that $L'_{+}(0)=0$, $\graph L\subset A_{r}^{u}$ and 
		$	M\cap A_{r}^{2u}=\epi L\cap A_{r}^{2u}$.
		\item[(v)]
		$M$ is a  {\it $\tilde T_{r,u}^5$-set} if there are DCR functions $U,L:[0,r)\to\er$ such that $L\leq U$ on $[0,r]$, $U'_{+}(0)=L'_{+}(0)=0$,
		$\graph U,\graph L\subset A_{r}^{u}$, and 
		$M\cap A_{r}^{2u}=\hyp U\cap\epi L$.
		\item[(vi)]
		$M$ is of type $T^i$ ($i=1,2,3,4,5$) at $x\in M$ in direction $v \in \sph^1$ if the preimage $(\gamma_{x,v})^{-1}(M)$ is a
		$\tilde T_{r,u}^i$-set for some $r,u>0$.
	\end{enumerate}
\end{definition}

Then we have (see \cite[Lemma~7.8]{PRZ18})
\begin{lem}\label{L:locpie}
	Let $M$ be a closed locally $\WDC$ set in $\er^2$, $x\in\partial M$ and $v\in \sph^1$.
	Then there exists $1\leq i \leq 5$ such that $M$ is of type $T^i$ at $x$ in direction $v$.
\end{lem}

For the purpose of characterising $\UWDC(\er^2)$ sets we will use the following version of the definition above. 
\begin{definition}\label{typesNew}
	Let $M \subset \R^2$ and $r, u>0$. We say that
	\begin{itemize}
	\item
	$M$ is a {\it $\tilde \T_{r,u}^1$-set} if $M\cap A_{r}^{2u}=\{0\}$.
	\item
	$M$ is a {\it $\tilde \T_{r,u}^2$-set} if $M\supset A_{r}^{2u}$.
	\item
	$M$ is a {\it $\tilde \T_{r,u}^3$-set} if $\partial M\cap A_r^{2u}\subset A_r^{u}$ and
	there are $n\in\en$ and DCR functions $f_1,\dots,f_n:[0,r)\to\er$ such that $f_i\leq f_{i+1}$ on $[0,r)$, $i=1,\dots,n-1$, $(f_i)'_+(0)=0$, $i=1,\dots,n$ and such that 
	\begin{equation}\label{eq:T3covering}
	\partial M\cap A_r^{u}=\bigcup_{i} \graph f_i
	\end{equation}
	and
	\begin{equation}\label{eq:T3sausages}
	\{x\in(0,r):\; f_{i}(x)=f_{i+1}(x)\}\not=\varnothing\;\implies\hyp f_{i+1}\cap\epi f_i\subset M,
	\end{equation}
	$i=1,\dots,n-1$.
	\item $M$ is of type $\T^i$ ($i=1,2,3$) at $x\in M$ in a direction $v \in \sph^1$ if the preimage $(\gamma_{x,v})^{-1}(M)$ is a 
	$\tilde \T_{r,u}^i$-set for some $r,u>0$.
\end{itemize}
\end{definition}

\begin{remark}\label{rem:oTypech}
	\begin{enumerate}[label={ (\alph*)}]
		\item\label{item:typeUnique}
		Clearly, if $M$ is a $\tilde \T_{r,u}^i$-set (resp. of  type $\T^i$  at $x$ in direction $v$), then $i$ is uniquely determined.
		\item\label{item:rFreeInType}
		If $M$ is a $\tilde \T_{r,u}^i$-set, then  $M$ is a $\tilde T_{r',u}^i$-set for every $r'<\delta$.
		\item\label{item:uFreeInType}
		If $M$ is a $\tilde \T_{r,u}^i$-set and $u>u'>0$ and $\delta>0$, then $M$ is a $\tilde T_{r',u'}^i$-set for some $r'<\delta$ (here we use that $(f_i)'_{+}(0)=0$ in Definition~\eqref{typesNew}).
	\end{enumerate}
\end{remark}

\begin{lem}\label{L:typeStableOnIntersection}
	Let $M$ be a closed set in $\er^2$, $x\in M$ and $v\in\sph^1$.
	Suppose that there are closed sets $M_j$, $j=1\dots,N$ with $M=\bigcup_{j=1}^N M_j$ and such that the set $M_I\coloneqq\bigcap_{j\in I} M_j$, $I\in\Sigma_N$, is of type $T_i$ at $x$ in direction $v$ for some $i=1,\dots,5$ for every $I\in\Sigma_N$.
	Then $M$ is of type $\T_i$ at $x$ in direction $v$ for some $i=1,2,3$.
\end{lem}

\begin{proof}
	We will prove the lemma by induction with respect to $N$.
	
	The case $N=0$ follows directly from what the definitions of the set of type $T_i$ and $\T_i$.
	So suppose that the statement of the lemma is true for $N$ up to some $n$,
	we will prove that it holds for $N=n+1$ as well.
	
	Without any loss of generality we can assume that $x=0$, $v=(1,0)$ and (by Remark~\ref{rem:oTypech}) that there are $r,u>0$ such that each set $M_I$ is a $\tilde T^i_{r,u}$-set for some $i$ (depending on $I$).
	We can also assume that each $M_j$ is a $\tilde T^i_{r,u}$-set for $i=3,4,5$. Indeed, if some $M_k$ was a $\tilde T^2_{r,u}$-set, then $M$ would be clearly a $\tilde \T^2_{r,u}$-set, and if some $M_k$ was a $\tilde T^1_{r,u}$-set, then $M\cap A_r^{2u}=\left(\bigcup_{j\not=k} M_j\right)\cap A_r^{2u}$ and we can use the induction procedure.
	
	For a simplicity of the notation we will assume that each $M_j$ is a $\tilde T^5_{r,u}$-set, the general case can be proved in a similar manner.
	By the definition a $\tilde T^5_{r,u}$-set we know that there are
	$U_i,L_i:[0,r)\to\er$, $i=1,\dots, N$ such that $L_i\leq U_i$ on $[0,r)$, $(U_i)'_{+}(0)=(L_i)'_{+}(0)=0$,
	$\graph U_i,\graph L_i\subset A_{r}^{u}$, and 
	$M\cap A_{r}^{2u}=\hyp U_i\cap\epi L_i$.
	We also know that each $M_I$ then has to be either a $\tilde T^1_{r,u}$-set, or a $\tilde T^5_{r,u}$-set, with the corresponding functions of the form $U=\min_{j\in I} U_j$ and $L=\max_{j\in I}L_j$.
	
	First note that we can suppose that there are $j,k\in\{1,\dots, N\}$, $j\not=k$, such that $M_{\{j,k\}}$ is a $\tilde T^5_{r,u}$-set.
	If not then we can reindex the sets $M_i$ isn such a way that
	\begin{equation}\label{eq:inequalitiesLU}
	L_1\leq U_1 < L_2\leq\cdots \leq U_{N-1}<L_N\leq U_N
	\end{equation}
	on $(0,r)$, in which case it is easy to verify that $M$ is a $\tilde\T^3_{r,u}$-set, where we define $f_{2j-1}\coloneqq L_j$ and $f_{2j}\coloneqq U_i$, $j=1,\dots N$ (note that \eqref{eq:T3sausages} follows from \eqref{eq:inequalitiesLU}).
	
	So assume that there are $j,k\in\{1,\dots, N\}$, $j\not=k$, such that $M_{\{j,k\}}$ is a $\tilde T^5_{r,u}$-set.
	Possibly reindexing the sets we can assume $j=n$ and $k=n+1$.
	Put $\tilde M_n =M_{n+1}\cup M_n$ and $\tilde M_l=M_l$, $l=1,\dots n-1$.
	We are done (by the induction procedure) if we can prove that each $\tilde M_I$ is a $\tilde T^1_{r,u}$-set or a $\tilde T^5_{r,u}$-set, $I\in\Sigma_n$.
	The only situation we need to check is if $n\in I$ and $|I|\geq 2$.
	Put $J=I\setminus\{n\}\in\Sigma_n$.
	Since $M_{\{n,n+1\}}$ is a $\tilde T^5_{r,u}$-set we know that 
	\begin{equation}\label{eq:inequalitioesUL}
	U_n\geq L_{n+1}\quad\text{and}\quad U_{n+1}\geq L_n,
	\end{equation}
	which, in particular, implies that $\tilde M_n$ is a $\tilde T^5_{r,u}$-set as well with the corresponding functions $U=\max (U_n, U_{n+1})$ and $L=\min (L_n, L_{n+1})$.
	
	If $\tilde M_J$ is a $\tilde T^1_{r,u}$-set we are done since $\tilde M_I$ is then a $\tilde T^1_{r,u}$-set as well.
	If $\tilde M_J$ is not a $\tilde T^1_{r,u}$-set it has to be a $\tilde T^5_{r,u}$-set.
	Put $\tilde U\coloneqq\min_{j\in J} U_j$ and $\tilde L\coloneqq\max_{j\in J}L_j$.
	Then
	\begin{equation}\label{eq:TilUsmallerTilL}
	M_J=\epi \tilde L\cap \hyp \tilde U\quad\text{and}\quad\tilde L\leq\tilde U.
	\end{equation}
	Now,
	\begin{equation}\label{eq:MIexpression}
	\tilde M_I=\tilde M_n\cap \tilde M_J=M_{J\cup \{n\}}\cup M_{J\cup \{n+1\}}.
	\end{equation}
	Hence, if either of the sets $M_{J\cup \{n\}}$ or $M_{J\cup \{n+1\}}$ is a $\tilde T^1_{r,u}$-set we are done.
	If both of them are $\tilde T^5_{r,u}$-sets then $L_{n},L_{n+1}\geq\tilde U$ and $U_{n},U_{n+1}\geq\tilde L$ $U\geq \tilde L$ which implies $\tilde U\geq L$ and so
	\begin{equation*}
	\tilde M_I=\tilde M_n\cap \tilde M_J
	=\epi \max(L,\tilde L)\cap \hyp \min(\tilde U,U).
	\end{equation*}
	Hence $M_I$ is a $\tilde T^5_{r,u}$-set which concludes the proof.
\end{proof}\color{black}

\color{black}
\begin{lem}\label{L:unionf_i=uniong_i}
	Let $r>0$ and let $f_1,\dots,f_n:(-r,r)\to \er$ be DC functions.
	Then there are DC functions $g_1,\dots,g_n:(-r,r)\to \er$ such that $g_i\leq g_{i+1}$ on $(-r,r)$, $i=1,\dots,n-1$ and such that
	\begin{equation}\label{eq:unionf_i=uniong_i}
	\bigcup_i \graph f_i=\bigcup_i \graph g_i.
	\end{equation}
\end{lem}

\begin{proof}
	For $x\in(-r,r)$ define the values $g_1(x),\dots,g_n(x)$ as follows.
	Let $\sigma$ be some permutation on $\{1,\dots,n\}$ such that
	\begin{equation}\label{eq:permutingf_i(x)}
	f_{\sigma(1)}(x)\leq f_{\sigma(2)}(x)\leq \cdots \leq f_{\sigma(n)}(x).
	\end{equation}
	Then we put $g_{i}(x)\coloneqq f_{\sigma(i)}(x)$.
	It is easy to see that the definition of $g_{i}(x)$ is independent of the choice of $\sigma$ and that \eqref{eq:unionf_i=uniong_i} holds, and it remains to show that the functions $g_i$ are DC on $(-r,r)$.
	By Lemma~\ref{L:Mixing} (and \eqref{eq:unionf_i=uniong_i}) it is enough to show that each $g_i$ is continuous.
	To do that fix $x\in(-r,r)$ and $i\in\{1,\dots,n\}$. 
	Let $\eps>0$. We want to find a $\delta>0$ such that $|g_i(x)-g_i(y)|<\eps$ whenever $|x-y|<\delta$ and $y\in(-r,r)$.
	Put 
	\begin{equation*}
	D\coloneqq\min\{|f_k(x)-f_l(x)|:\;  k,l=1,\dots,n,\; f_k(x)\not = f_l(x)\}.
	\end{equation*}
	By the continuity of each $f_i$ we can find $\delta_1>0$ such that $|f_k(x)-f_k(y)|<\frac{D}{2}$ whenever $|x-y|<\delta_1$ and $k\in\{1,.\dots,n \}$.
	Similarly, we can find $\delta_2>0$ such that $|f_k(x)-f_k(y)|<\eps$ whenever $|x-y|<\delta_2$ and $k\in\{1,.\dots,n \}$.
	Put $\delta\coloneqq\min(\delta_1,\delta_2)$.
	
	Suppose that $|x-y|<\delta$ and let $j$ be such that $g_i(y)=f_j(y)$.
	We first claim that $g_i(x)=f_j(x)$. 
	To prove the claim consider $l$ such that $ f_j(x)>f_l(x)$ then (using $\delta\leq\delta_1$) we obtain
	\begin{equation*}
\begin{aligned}
	D&\leq f_j(x)-f_l(x)= f_j(x)-f_j(y)+f_j(y)-f_l(y)+f_l(y)-f_l(x)\\
	&<2\cdot\frac{D}{2}+f_j(y)-f_l(y).
\end{aligned}
	\end{equation*}
	Similarly
	$D <D+f_l(y)-f_j(y)$, provided $ f_j(x)<f_l(x)$.
	Therefore
	\begin{equation*}
	f_j(x)>f_l(x)\implies f_j(y)>f_l(y)\quad\text{and}\quad
	f_j(x)<f_l(x)\implies f_j(y)<f_l(y).
	\end{equation*}
	Since there is at most $i-1$ indices $l$ such that $g_i(y)=f_j(y)>f_l(y)$ we obtain that there is at most $i-1$ indices $l$ such that $f_j(x)>f_l(x)$ and so we know (using \eqref{eq:permutingf_i(x)}) that $g_i(x)\geq f_j(x)$ and
	similarly, there is at most $n-i$ indices $l$ such that $f_j(x)<f_l(x)$ and so $g_i(x)\leq f_j(x)$. Hence, $g_i(x)=f_j(x)$.
	Using the claim and also the fact that $\delta\leq\delta_2$ we can write
	\begin{equation*}
	|g_i(x)-g_i(y)|=|f_j(x)-f_j(y)|<\eps,
	\end{equation*}
	which completes the proof of the lemma.
\end{proof}\color{black}
\color{black}
\begin{lem}\label{L:boundaryImpliesType}
	Let $M\subset\er^2$ be a compact set such that $M^c$ has finitely many components and $\partial M$ is a union of finitely many DC graphs. Suppose that $x\in\partial M$  and $v\in \sph^1$. Then $M$ is of type $\T^i$ at $x$ in direction $v$ for some $i=1,2,3$.
\end{lem}

\begin{proof}
	Without any loss of generality we may assume that $x=0$ and $v=(0,1)$.
	
	By the assumptions of the lemma, $\partial M$ is a union of finitely many DC graphs $P_1,\dots,P_n$.
	Put $I=\{i: 0 \in P_i \}$.
Clearly $I \not= \varnothing$ ($x\in\partial M$), so we can suppose that $I= \{1,\dots, N\}$ for some $N\in\en$. 
	
	Put $\tilde I=\{i: \Tan(P_i,0)\supset \{(1,0)\} \}$.
	If $\tilde I = \varnothing$, then there exist $r, u >0$ such that  $\partial M\cap A_{r}^{2u}=\{0\}$
	and so $M$ is either a $\tilde{\T}_{r,u}^1$-set or a a $\tilde{\T}_{r,u}^2$-set.
	If $\tilde I \neq \varnothing$, we can again suppose that $\tilde I= \{1,\dots, \tilde N\}$ for some $\tilde N\in\en$. 
	We will prove that $M$ is an $\tilde{\T}_{r,u}^3$-set for some $u,r>0$.

	Due to Remark~\ref{rem:DCgraphProperties} we can find DC graphs $Q_1,\dots,Q_{\tilde N}$ such that 
	\begin{equation}\label{eq:TanIsSpan}
	\Tan(Q_i,0)=\spa \{(1,0)\}
	\end{equation}
	and such that 
	\begin{equation}\label{eq:Qi=Pi}
	Q_i\cap\{(x,y:x\geq 0)\}=P_i\cap\{(x,y:x\geq 0)\},\quad i=1,\dots, \tilde N.
	\end{equation}
	
	Using Lemma \ref{L:zestr} we obtain that  there exist 
	$u, \rho\in (0,\infty)$ such that for each $1\leq i \leq \tilde N$ there is a DCR function $\vf_i$ on $(-\rho, \rho)$ such that
	$Q_i \cap A_{r}^{2u} = \graph \vf_i$ for every $0<r<\rho$. 
	Note that also $(\vf_i)'(0) = 0$ by \eqref{eq:TanIsSpan}.
	Using \eqref{eq:Qi=Pi} we obtain 
	\begin{equation*}\label{eq:coverOfIntersection}
	\partial M\cap  A_{r}^{2u} = \left(\bigcup_{i}^{\tilde N}\graph \vf_i \right)\cap  A_{r}^{2u}
	\end{equation*}
	for every $0<r<\rho$.
	Since $(\vf_i)'(0) = 0$ we may additionally assume (perhaps by making $\rho$ smaller) that $\partial M\cap  A_{r}^{2u}\subset A_{r}^{u}$ for every $0<r<\rho$.
	Moreover, using Lemma~\ref{L:unionf_i=uniong_i} we can find DCR functions  $f_1,\dots,f_{\tilde N}:(-\rho,\rho)\to\er$ such that 
	\begin{equation*}
	\bigcup_{i}^{\tilde N}\graph \vf_i =\bigcup_{i}^{\tilde N}\graph f_i,
	\end{equation*}
	$f_i'(0)=0$
	and that $f_i\leq f_{i+1}$, $i=1,\dots,\tilde N-1$.
	It remains to prove that \eqref{eq:T3sausages} holds for some $0<r<\rho$.
	To do that put
	\begin{equation*}
	H_i\coloneqq\{x\in(0,\rho):\; f_{i+1}(x)=f_i(x)\}
	\end{equation*}
	and suppose for a contradiction that for some $i$ condition \eqref{eq:T3sausages} does not hold for any $0<r<\rho$.
	Fix some such $i$.
	Then there are $x_j,y_j\in H_i$, $j\in\en$, satisfying $y_{j+1}<x_j<y_j$ and $f_{i+1}>f_i$ on $(x_j,y_j)$, $j\in\en$.
	But this is a contradiction with the assumption that $M^c$ has only finitely many connected components since each set of the form $\{(x,y):\; x\in(x_j,y_j),\,f_{i+1}>y>f_i \}$, $j\in\en$, is a connected component of $M^c$.
\end{proof}

\color{black}

\begin{lem}\label{L:subEpiAura}
	Let $a>0$ and let $g,h$ be two DCR functions on $[0,a]$, $g\geq h$ on $[0,a]$.
	Then the set $\subgr g\cap\epi h\subset\er^2$ is $\WDC$.
\end{lem}

\begin{proof}
	Pick $L>0$ such that both $g$ and $h$ are $L$-Lipschitz on $[0,a]$.
	Define functions $\tilde g,\tilde h:\er\to\er$ by $\tilde{g}(x)=g(x)$ and $\tilde{h}(x)=h(x)$, $x\in[0,a]$, $\tilde{g}(x)=g(a)$ and $\tilde{h}(x)=h(a)$, $x\in[a,\infty)$ and
	$\tilde{g}(x)=g(0)$ and $\tilde{h}(x)=h(0)$, $x\in(-\infty,0]$.
	Then $\tilde g$ and $\tilde h$ are $L$-Lipschitz DC functions on $\er$ such that $\tilde g|_{[0,a]}=g$, $h|_{[0,a]}=h$ and $\tilde{g}\leq\tilde{h}$ on $\er$.
	Denote $M\coloneqq\subgr g\cap\epi h$, $A=[0,a]\times\er$ and  $B\coloneqq\subgr \tilde g\cap\epi\tilde h$.
	Clearly $M=A\cap B$.
	
	Define, for $(x,y)\in\er^2$, $F(x,y)\coloneqq 2L\max (0,x-a,-x)$,
	$G(x,y)\coloneqq\max(0,y-g(x),h(x)-y)$.
	and $H\coloneqq F+G$.
	We will prove that $H$ is a DC aura for $M$.
	It is easy to see that $H$ is DC and that $M= H^{-1}(\{0\})$.
	It remains to show that $0$ is a weakly regular value of $H$.
	
	Note that if $(x,y)\in \er^2\setminus M$, then one of the conditions $x>a$, $x<0$, $y>g(x)$ or $y<h(x)$ holds.
	Pick $(x,y)\in \er^2\setminus M$ and $v\in\partial H(x,y)$.
	Now, if $x>0$ then $v_1\geq L$, if $x<0$ then $v_1\leq -L$, if $y>g(x)$ then $v_2= 1$ and, finally, if $y<h(x)$ then $v_2= -1$.
	In either case $|v|\geq\min(1,L)>0$ which proves that $0$ is a weakly regular value of $H$.
\end{proof}

\color{black}

\begin{definition}
	A set $M\subset\er^2$ is called a DC cone if there are $a>0$, two DCR functions on $[0,a]$, $g\geq h$ on $[0,a]$, and a rotation $\gamma$ on $\er^2$ such that $M=\gamma(\subgr g\cap\epi h)$.
\end{definition}



\begin{cor}\label{C:typeImpliesUWDC}
	Let $M$ be a set in $\er^2$ such that for every $x\in\partial M$ and $v\in \sph^1$, there exists $1\leq i \leq 3$ such that $M$ is of type $\T^i$ at $x$ in direction $v$.
	Then $M$ is $\UWDC$.
\end{cor}

\begin{proof}\color{black}
	Pick $x\in M$. We want to find a $\UWDCG$ set $K\subset\er^2$ and $\rho>0$ such that $M\cap B(x,\rho)=K\cap B(x,\rho)$.
	This is enough to prove that $M$ is $\UWDC$ by Remark~\ref{rem:UWDSremarks}~\ref{item:remarkIntersection}.
	
	First of all we can assume that $x\in\partial M$ (since otherwise $x\in M^{\circ}$ and it is enough to pick any $\rho>0$ such that $B(x,\rho)\subset M$ and $K=B(x,\rho)$).
	
	First we claim that there are finitely many $v_1,\dots, v_N\in\sph^1$ and $u,r>0$ such that 
	\begin{enumerate}[label={ (\Alph*)}]
		\item\label{cond:TiForEveryv}  $M$ is a $\tilde \T^3_{r,u}$-set at $x$ in direction $v_j$ for every $j=1,\dots,N$,
		\item\label{cond:A2urCoverBall} 
		\begin{equation*}
		\partial M\cap B\left(x,\frac{r}{2}\right)\subset \bigcup_{j=1}^N A^{u}_r(x,v_j)
		\end{equation*}
		\item\label{cond:A2urDontIntersect} $A^{2u}_r(x,v_j)\cap A^{2u}_r(x,v_k)=\{x\}$ whenever $j\not=k$.
	\end{enumerate}


Let $V_3\subset \sph^1$, be the set of all $v$ such that $M$ is of type $\T^3$ at $x$ in the direction $v$. Note that the set $V_3$ is finite (if not, then there would be a sequence $\{v_i\}\subset V_3$ converging to some $v\in \sph^1$, but then $M$ cannot be of type $\T^i$ at $x$ in direction $v$ for any $i=1,2,3$).
Let $V_3=\{v_1,\dots,v_N\}$ and let $r_l$ and $u_l$, $l=1,\dots, N$ be such that
$(\gamma_{x,v_l})^{-1}(M)$ is a $\tilde \T_{r_l,u_l}^3$-set.
By Remark~\ref{rem:oTypech}, \ref{item:rFreeInType} and \ref{item:uFreeInType} we may assume that there are some $r',u>0$ such that $r_l=r'$ and $u_l=u$, $l=1,\dots, N$, and also that \ref{cond:TiForEveryv} and \ref{cond:A2urDontIntersect} hold for every $0<r<r'$.


It remains to prove \ref{cond:A2urCoverBall}. To do this assume (aiming to a contradiction) that for every $n\in\en$ there is $y_n\in\partial M\cap B(x,\frac1n)$ such that
\begin{equation}\label{eq:y_nnotInUnion}
y_n\not\in\bigcup_{j=1}^N A^{u}_r(x,v_j).
\end{equation}
Define $w_n\coloneqq \frac{y_n-x}{|y_n-x|}$. We can assume (possibly by passing to a subsequence) that $w_n\to w\in\sph^1$ as $n\to\infty$.
Then one one hand $w\in V_3$ ($A^u_\rho(x,w)\cap\partial M\not=\varnothing$ for any $u,\rho>0$), but on the other hand \eqref{eq:y_nnotInUnion} implies that $w\not\in V_3$, a contradiction.

%

Pick now $v_1,\dots, v_N\in\sph^1$ and $u,r>0$ as in the claim above.
For every $i=1,\dots,N$ we have
\begin{equation*}
M\cap \overline{A^{2u}_r(x,v_j)}=K^{-}_j\cup K^{+}_j\cup K^0_j,
\end{equation*}
where

\begin{equation*}
K_j^+\coloneqq  M\cap \gamma_{x,v_j}(\left\{(z,y): 0\leq z\leq r, zu\leq y\leq 2zu \right\}),
\end{equation*}
\begin{equation*}
K_j^-\coloneqq  M\cap \gamma_{x,v_j}(\left\{(z,y): 0\leq z\leq r, -zu\geq y\geq -2zu \right\})
\end{equation*}
and
\begin{equation*}
K_j^0\coloneqq  M\cap \overline{A^{u}_r(x,v_j)}.
\end{equation*}

Put 
\begin{equation*}
M_3\coloneqq \left(M\cap B\left(x,\frac{r}{2}\right)\right)\setminus  \bigcup_{j=1}^q (A^{2u}_r(x,v_j))^\circ.
\end{equation*}
Then it is easy to see that (by \ref{cond:A2urCoverBall} and \ref{cond:A2urDontIntersect}) $M_3$ is a union of finitely many closed convex circular sectors $S_1\dots, S_p$.

First note that each set $K_j^\pm$ is convex (in fact it is always either a singleton $\{x\}$, or a triangle with vertices $x$, $\gamma_{x,v_j}(r,\pm ur)$ and $\gamma_{x,v_j}(r,\pm 2ur)$).
Moreover, by \ref{cond:A2urDontIntersect},
\begin{equation}\label{eq:K0Kpm}
K_j^0\cap K_l^\pm=\{x\}, \quad j\not=l
\end{equation}
and
\begin{equation}\label{eq:K0K0}
K_j^0\cap K_l^0=\{x\}, \quad j\not=l.
\end{equation}
Also,
\begin{equation}\label{eq:type12Convex2}
S_i\cap K^0_l=\{x\}, \quad \text{ for every $i=1,\dots,p$ and $l=1,\dots,N$}.
\end{equation}
Next we claim that for every $j\in\{1,\dots, N\}$ there are $q(j)\in\en$ and $K_{j,m}\subset\er^2$, $m=1,\dots,q(j)$ such that
\begin{enumerate}[label={\rm (\roman*)}]
	\item\label{cond:KjmWDC} each $K_{j,m}$ is $\WDC$,
	\item\label{cond:KjmUnion} $K^0_j=\overline{A^{u}_r(x,v_j)}\cap\bigcup_{m=1}^{q(j)}K_{j,m}$,
	\item\label{cond:KjmIntersections} $K_{j,m}\cap K_{j,l}=\{x\}$, $m\not=l$,
	\item\label{cond:KjmintersectionConvex} $K_{j,m}\cap K^\pm_j$ is convex for each $m$.
\end{enumerate}
To prove the claim pick some such $j$.
We can assume that $v_j=(1,0)$ and so $M$ is a $\tilde \T^3_{r,u}$-set.
Let $f_1,\dots f_n$ be the corresponding functions from the definition of $\tilde \T^3_{r,u}$-set.
Put $f_0(x)\coloneqq-ux$ and $f_{n+1}(x)\coloneqq ux$.

Denote $g_i$ as a continuous extension of $f_i$ (which is defined on $[0,r)$) to $[0,r]$, $n=0,\dots,n+1$. Clearly each $g_i$ is DCR on $[0,r]$.
Let $I$ be the system of those $m\in\{0,\dots,n\}$ that satisfy $\{(a,b)\in M: f_m(a)<b<f_{m+1}(a) \}\not=\varnothing$ and let $J$ be the system of all $m\in\{1,\dots,n\}$ satisfying both $m-1\notin I$ and $m\notin I$.
Put $q(j)=|J|+|I|$. 
Pick some bijections $\kappa:\{1,\dots,|I|\}\to I$ 
and $\tau:\{1,\dots,|J|\}\to J$ and
define $K_{j,m}\coloneqq\graph g_{\tau(l)}$, $l=1,\dots, |J|$ and
\begin{equation*}
K_{j,|J|+l}\coloneqq \epi g_{\kappa(l)}\cap \subgr g_{\kappa(l)+1},
\quad l=1,\dots, |I|.
\end{equation*}

Now, Lemma~\ref{L:subEpiAura} implies \ref{cond:KjmWDC}, \eqref{eq:T3covering} implies \ref{cond:KjmUnion} and \ref{cond:KjmIntersections} follows from \eqref{eq:T3sausages}.
Also, since $K_{j,m}\cap K^\pm_j$ is either a line segment (this can only happen when $\kappa(m)=n$ or $\kappa(m)=0$) or a singleton $\{x\}$ (in all other cases) we obtain \ref{cond:KjmintersectionConvex}.

Put
\begin{equation*}
\K^s=\{K^\pm_1,\dots,K^\pm_N\},\quad
\sS=\{S_i:\; i=i,\dots p \},
\end{equation*}
\begin{equation*}
\K_3=\{K_{j,m}:\; j=1,\dots, N,\; m=1,\dots, q(j),\},
\end{equation*}
\begin{equation*}
\K=\K^s\cup\sS\cup\K^0_3\quad\text{and}\quad K=\bigcup \K.
\end{equation*}

Note that both $\K^s$ and $\sS$ are collections of convex sets while $\K_3$ is a collection of $\WDC$ sets. We claim that $K$ is $\UWDCG$. This is enough to finish the proof ot the theorem since $K\cap U\left(x,\frac{r}{2}\right)=M\cap U\left(x,\frac{r}{2}\right)$.

Pick $\varnothing\not=\K'\subset\K$ and put $K'=\bigcap \K'$.
We want to prove that $K'$ is $\WDC$.
First of all, if $|\K'|=1$, then we are done since $K'$ is either convex (if $\K'\subset \K^s\cup\sS$) or $\WDC$ by \ref{cond:KjmWDC}.
If $|\K'|>1$ then $K'$ is always convex by \eqref{eq:K0Kpm}, \eqref{eq:K0K0}, \eqref{eq:type12Convex2}, \ref{cond:KjmIntersections} and \ref{cond:KjmintersectionConvex} (recall that $\sS$ and $\K^s$ contain only convex sets).
Of course, each convex set is $\WDC$.
\end{proof}

\color{black}
\begin{thm}\label{T:characterizationUWDC}
	Let $M\subset\er^2$ be a compact set. Then the following conditions are equivalent.
		\begin{enumerate}[label={ (\Alph*)}]
		\item\label{cond:MisUWDC} $M\in\UWDC$,
		\item\label{cond:MdcBoundary} $M^c$ has finitely many connected components and $\partial M$ is a union of finitely many DC graphs,
		\item\label{cond:Mpie} for every $x\in\partial M$ and $v\in \sph^1$, there exists $1\leq i \leq 3$ such that $M$ is of type $\T^i$ at $x$ in direction $v$.
	\end{enumerate}
\end{thm}

\begin{proof}
	We start with the implication $\ref{cond:MisUWDC}\!\!\implies\!\!\ref{cond:Mpie}$.
	Pick $x\in \partial M$ and $v\in\sph^1$.
	By the definition there is a neighbourhood $U$ of $x$ and sets $M_1,\dots,M_{j}$ such that
	$M\cap U=U\cap\bigcup_{i=1}^{j} M_i$ and such that each set $M_{I}:=\bigcap_{i\in I} M_i$, $I\in\Sigma_j$ is $\WDC$.
	By \cite[Lemma~7.8]{PRZ18} we know that each $M_{I}$ is of type $T_j$ for some $j=1,\dots,5$ and so $M$ is of type $\T^i$ at $x$ in direction $v$ (for some $i\in\{1,2,3\}$) by Lemma~\ref{L:typeStableOnIntersection} and we are done.
	
	The implication $\ref{cond:Mpie}\!\!\implies\!\!\ref{cond:MisUWDC}$ follows from Corollary~\ref{C:typeImpliesUWDC} and the implication $\ref{cond:MdcBoundary}\!\!\implies\!\!\ref{cond:Mpie}$ follows directly from Lemma~\ref{L:boundaryImpliesType}.

	Now we prove the implication $\ref{cond:Mpie}\!\!\implies\!\!\ref{cond:MdcBoundary}$.
%
	 
	 Analogously to the proof of Corollary~\ref{C:typeImpliesUWDC} we can find for every $x\in \partial M$ some $u(x),r(x)>0$, $N(x)\in\en$ and $v^x_1,\dots,v^x_{N(x)}\in\sph^1$ such that 
	 \begin{enumerate}[label={ (\alph*)}]
	 	\item\label{cond:TiForEveryv2} for every $j=1,\dots,N(x)$ there is $i\in\{1,2,3\}$ such that $(\gamma_{x,v^{x}_j})^{-1}(M)$ is a $\widetilde{\T}^i_{r(x),u(x)}$-set at $x$ in direction $v^x_j$,
	 	\item\label{cond:A2urCoverBall2}
	 	\begin{equation*}
	 	B(x,r(x))\subset \bigcup_{j=1}^{N(x)} A^{2u(x)}_{r(x)}(x,v^x_j).
	 	\end{equation*}
	 \end{enumerate}

	 Since $\partial M$ is compact we know that there are $x_1,\dots, x_p$ such that $\partial M$ is covered by balls $B(x_i,r(x_i))$, $i=1,\dots,p$.
	 Therefore there is some $\rho>0$ such that the parallel set $(\partial M)_{\rho}$ is also covered by balls $B(x_i,r(x_i))$, $i=1,\dots,p$.
	 
	 By \ref{cond:A2urCoverBall2} we have that $(\partial M)_{\rho}$ is therefore covered by the system 
	 \begin{equation*}
	 \A\coloneqq\left\{A^{2u(x_l)}_{r(x_l)}(x_l,v^{x_l}_j):\;
	 l=1,\dots,p,\;\; j=1,\dots,N(x_l)\right\}.
	 \end{equation*}
	 
	 First note that $\A$ is a finite cover of $\partial M$ and that (by \ref{cond:TiForEveryv2}) $\partial M\cap A$ is a union of finitely many DC graphs for every $A\in\A$ and so $\partial M$ is a union of finitely many DC graphs as well.
	 
	 Moreover, by \ref{cond:TiForEveryv2} we also know that $A\setminus M$ has only finitely many connected components relatively in $A$ for each $A\in\A$.
	 
	 Since $\A$ is a finite covering of $\bigcup\A$ we know by Proposition~\ref{P:numberOfComponentsBound} that the number of connected components of $\bigcup \A$ (relatively in $\bigcup\A$) is finite.
	 And, finally, since $\bigcup\A$ contains a neighbourhood of $\partial M=\partial (M^c)$ 
	 we know that each connected component of $M^c$ has nonempty intersection with $\bigcup\A$ and so the number of connected components of $M^c$ is finite by Proposition~\ref{P:numberOfComponentsBound2} .
\end{proof}\color{black}

\section{Nested sets in plane}\label{sec:nestedSets}
\color{black}
\begin{lem}\label{L:hausdorffLimitOfUnion}
Let $\{A_i^k\}_{k=1}^\infty$, $i=1,\dots,N$, be sequences of a nonempty compact sets in $\er^d$ and suppose that $A_i^k\to A_i$ as $k\to \infty$ (in the Hausdorff distance) for some $A_i\subset\er^d$, $i=1,\dots,N$.
Put
\begin{equation*}
A^k\coloneqq\bigcup_{i=1}^N A_i^k,\; k\in\en\quad\text{and}\quad A\coloneqq\bigcup_{i=1}^N A_i.
\end{equation*}
Then $A^k\to A$ as $k\to \infty$.
\end{lem}

\begin{proof}
	Pick $\eps>0$. We need to prove that for some $m\in\en$, $\dist_{\H}(A^k,A)\leq\eps$ whenever $k\geq m$ which is the same that $A^k\subset A_{\eps}$ and $A\subset (A^k)_{\eps}$ whenever $k\geq m$.
	Since $A_i^k\to A_i$ there is some $m_i\in\en$ such that 
	$A_i^k\subset (A_i)_{\eps}$ and $A_i\subset (A_i^k)_{\eps}$ whenever $k\geq m_i$. 
	Therefore if $k\geq m\coloneqq \max_{i}m_i$ we have
	\begin{equation*}
	(A^k)_\eps=\biggl(\bigcup_{i=1}^N A_i^k\biggr)_\eps
	=\bigcup_{i=1}^N (A_i^k)_\eps\supset \bigcup_{i=1}^N A_i=A
	\end{equation*}
	and
	\begin{equation*}
	A^k=\bigcup_{i=1}^N A_i^k
	\subset\bigcup_{i=1}^N (A_i)_\eps= \biggl(\bigcup_{i=1}^N A_i\biggr)_\eps=A_\eps,
	\end{equation*}
	which concludes the proof
\end{proof}

\begin{lem}\label{L:limitOfGraphs}
	Let $L>0$ and let $f_i:[a_i,b_i]\to [0,\infty]$ be $L$-Lipschitz $\C^2$ functions such that $\int_{a_i}^{b_i}|f_i''(t)|\;dt<L$, $i\in\en$.
	Suppose that there is a non-empty compact set $M\subseteq \er^2$ such that $\graph f_i\to M$ in the Hausdorff distance as $\i\to\infty$.
	Then there are $-\infty<a\leq b<\infty$ and a DCR function $f$ on $[a,b]$ such that $M= \graph f$.
\end{lem}

\begin{proof}
	First note that since $\graph f_i\to M$ in the Hausdorff distance and since all sets $\graph f_i$ are connected, we have that $[a_i,b_i]=\Pi_{V}(\graph f_i)\to \Pi_{V}(M)\eqqcolon [a,b]$ in the Hausdorff distance, where $V\coloneqq \spa ((1,0))$. 
	In particular, $a_i\to a$ and $b_i\to b$.
	Since all $f_i$ are Lipschitz with the same constant we easily obtain that $M_x$ is a singleton for every $x\in[a,b]$.
	So we can define $f:[a,b]\to\er$ in such a way that $\{y: (x,y)\in M \}=\{f(x)\}$, $x\in[a,b]$.
	Therefore $M=\graph f$ and it remains to prove that $f$ is DCR on $[a,b]$.
	The trivial case $a=b$ is obvious and so we will assume $a<b$.
	
	It is not difficult to see that we can extend/restrict functions $f_i$ to obtain $L$-Lipschitz $\C^2$ functions $\tilde f_i:[a,b]\to \er$ such that $\graph\tilde f_i\to \graph f$ and such that $\int_{a}^{b}| (\tilde f_i)''(t)|\;dt<L$, $i\in\en$.
	Note that $\graph\tilde f_i\to \graph f$ implies that $\tilde f_i\to f$ uniformly on $[a,b]$. 
	
	Define for $x\in[a,b]$
	\begin{equation*}
	g_i(x)\coloneqq \int_a^x\int_a^t ( (\tilde f_i)''(s))_+\,ds\,dt\quad\text{and}\quad 
	h_i(x)\coloneqq \int_a^x\int_a^t ((\tilde f_i)''(s))_-\,ds\,dt.
	\end{equation*}
	Then $f_i=g_i-h_i+\phi_i$, where $\phi_i(x)=(f_i)_{+}'(a)\cdot(x-a)+f_i(a)$.
	Moreover, $g_i$ and $h_i$ are convex and $(b-a)L$-Lipschitz $\C^2$ functions on $[a,b]$ and $\phi_i$ is $L$-Lipschitz affine function on $[a,b]$, $i\in\en$.
	Additionally, $g_i(a)=h_i(a)=0$ and $\phi_i(a)\to f(a)$ and so the sequences $\{g_i\}$, $\{h_i\}$ and $\{\psi_i\}$ are uniformly bounded.
	Hence, by a standard (multiple) application of the Arzel\`a-Ascoli theorem, we can find convergent subsequences
	$\{g_{i_k}\}$, $\{h_{i_k}\}$ and $\{\phi_{i_k}\}$ converging to functions $g$, $h$ and $\psi$, respectively. Clearly, $g$ and $h$ are convex Lipschitz, and $\phi$ is affine.
	Since $g_{i_k}-h_{i_k}+\phi_{i_k}=f_{i_k}\to f$ as $k\to\infty$ we obtain that $f=g-h+\phi$ and so $f$ is DCR.
\end{proof}


%

\begin{prop}\label{prop:Ner2SubsetUWDC}
	$\N(\er^2)\subset\UWDC(\er^2)$.
\end{prop}

\begin{proof}
	Pick $M\in\N(\er^2)$.
By Theorem~\ref{T:characterizationUWDC} it is enough to prove that $M^c$ has finitely many connected components and that $\partial M$ is a union of finitely many DC graphs.

Let $M_i$, $i\in\en$, be compact $\C^2$ smooth domains such that 
\begin{equation}\label{E:intersectionOfM_i}
\bigcap_{i=1}^\infty M_i=M 
\end{equation}
and such that
\begin{equation}\label{E:boundOnMassOfM_i}
\bM(N_{M_i})\leq L<\infty.
\end{equation}

First observe that by \eqref{E:boundOnMassOfM_i} there is $N\in\en$ such that $\widetilde C_0^{var}(M_i)\leq N$, $i=1,\dots$, which, in particular, implies that
\begin{equation}\label{E:C_iBounded}
C_i\leq N,\quad i=1,\dots,
\end{equation}
where $C_i$ is the number of connected components of $\partial M_i$.
Since the number of connected components of $M_i^c$ is bounded by $C_i$ and therefore by $N$ we obtain (using \eqref{E:intersectionOfM_i})
that $M^c$ has finitely many connected components.

It remains to prove that $\partial M$ is a union of finitely many DC graphs.
By \eqref{E:C_iBounded} we can find constant subsequence $\{C_{i_k}\}$.
Let $C$ be the constant value of that subsequence and let $M_{k}^l$ be the components of $M_{i_k}^c$, $k=1,\dots$, $l=1,\dots,C$.

Next we claim that the boundary of each $M_{k}^l$ is a union of at most $5N$ $1$-Lipschitz graphs.

To do that pick some $k$ and $l$ as above and put $S\coloneqq\partial M_{k}^l$.
Then there is a $1$-periodic $\C^2$ curve $\gamma:\er\to\er^2$ such that $S=\Im(\gamma)$ and that $\gamma|_{[0,1)}$ is injective.
Let $\nu:\er\to\sph^1$ be such that $\nu(t)$ is a outer normal to $M_{k}^l$ at $\gamma(t)$.
Note that 
\begin{equation}\label{eq:CvarGeqDist}
\widetilde C^{var}_0(M_{k}^l,\gamma((a,b)))\geq \frac{\H^1(\nu((a,b)))}{2\pi}\geq \frac{\rho(\nu(a),\nu(b))}{2\pi}
\end{equation}
for any $a\leq b$.

Put $t_1\coloneqq 1$ and define 
\begin{equation*}
t_{j+1}\coloneqq \min\left\{t\geq t_j:\rho(\nu(t_k),\nu(t))\geq\frac{\pi}{2}\right\},\quad j\in\en,\,j\geq 2.
\end{equation*}
Note that we can take the minimum in the above definition since $\nu$ is continuous.

Clearly $\rho(\nu(t_k),\nu(t_{k+1}))=\frac{\pi}{2}$, $k\in\en$, and so \eqref{eq:CvarGeqDist} implies that
\begin{equation}\label{eq:curvatureOfPiecesOfGamma}
N\geq \widetilde C^{var}_0(M_{k}^l,\gamma((0,1)))
\geq \sum_{j=1}^{m}\widetilde C^{var}_0(M_{k}^l,\gamma((t_j,t_{j+1})))
\geq\frac{m}{4},
\end{equation}
provided $t_{m+1}\in[0,1)$.
Put $G^{k}_{l,j}\coloneqq\gamma([t_j,t_{j+1}])$, clearly $S\subset \bigcup_{j=1}^{5N} G^{k}_{l,j}$ by \eqref{eq:curvatureOfPiecesOfGamma}.

To finish the proof of the claim it is enough to prove that $G^{k}_{l,j}$ is always a $1$-Lipschitz graph.
To do this pick $1\leq j\leq 5N$ and let $\tilde\nu$ be the middle point of the shortest arc connecting $\nu(t_{j})$ and $\nu(t_{j+1})$.
Then $\rho(\tilde\nu,\nu(t))\leq \frac{\pi}{4}$ for every $\nu\in\nu([t_j,t_{j+1}])$.
Without the loss of generality we can assume that $\tilde\nu=(0,1)$.
Pick $(\sin\alpha,\cos\alpha)=\nu(t)$ for some $t\in[t_j,t_{j+1}]$ and suppose that $\gamma=(\gamma_1,\gamma_2)$ is parametrized by the arc length.
Then $\gamma'(t)=(\cos\alpha,-\sin\alpha)$ with $\alpha\in[-\frac{\pi}{4},\frac{\pi}{4}]$ and so
\begin{equation}\label{eq:lipConstEstimate}
\left|\frac{\gamma_2'(t)}{\gamma_1'(t)}\right|=|\tan\alpha|\leq 1.
\end{equation}
Define $f(u)=\gamma_2(\gamma_1^{-1}(u))$, $u\in[\gamma_1(t_j),\gamma_1(t_{j+1})]$.
The function $f$ is well defined since by the above $\gamma_1$ is increasing.
Easy computation gives us $f'(u)=\frac{\gamma_2'(\gamma_1^{-1}(u))}{\gamma_1'(\gamma_1^{-1}(u))}$ and so by \eqref{eq:lipConstEstimate} we have $|f'|\leq 1$ and so $f$ is $1$-Lipschitz on $[\gamma_1(t_j),\gamma_1(t_{j+1})]$.
Clearly $G^{k}_{l,j}=\graph f$ and the proof of the claim is finished.

Define $A^{k}_{m}$, $k\in\en$, $m=1,\dots,5NC$, by
\begin{equation*}
A^{k}_{5N(l-1)+j}\coloneqq G^{k}_{l,j},\quad l=1,\dots,C,\; j=1,\dots, 5N.
\end{equation*}
By the construction 
$ A^k\coloneqq\partial M_{i_k}=\bigcup_{m=1}^{5NC} A^{k}_{m}$.
Passing to a subsequence (at most $5NC$ times) we may assume that each sequence $\{A^{k}_{m}\}$, $m=1,\dots, 5NC$, converges in the Hausdorff distance with a limit $A_m$ as $k\to \infty$.
We also have $\partial M_{i_k}\to \partial M$
and so Lemma~\ref{L:hausdorffLimitOfUnion} implies that $M=\bigcup_{m=1}^{5NC} A_m$.
Therefore to complete the proof of the lemma it is enough to prove that each $A_m$ is a DC graph.

Fix $m\in\{1,\dots,5NC\}$. First we prove that there is some $v\in\sph^{1}$ and a subsequence $A^{k_l}_{m}$ such that $A^{k_l}_{m}$ is a $2$-Lipschitz graph in the direction $v$ for every $l$.

Let $v_m \in\sph^{1}$ be such that $A_m^k$ is a $1$-Lipschitz graph in the direction $v_k$. Then there is a subsequence $v_{k_l}$ and some $v\in\sph^{d-1}$ such that $|v-v_{k_l}|\leq \frac{1}{10}$, $l\in\en$.
Then by Lemma~\ref{L:LipschitzGraphEquivalence}
\begin{equation*}
\begin{aligned}
|(A-B)\cdot v|&\leq |(A-B)\cdot v_{k_l}|+|(A-B)\cdot (v-v_{k_l})|\\
&\leq \frac{1}{\sqrt{2}}|A-B|+|A-B|\cdot|v-v_{k_l}|
\leq \left(\frac{1}{\sqrt{2}}+\frac{1}{10}\right)|A-B|\\
&\leq \frac{2}{\sqrt{5}}|A-B|
\end{aligned}
\end{equation*}
whenever $A,B\in A_{m}^{k_l}$,
and so (again by Lemma~\ref{L:LipschitzGraphEquivalence})
each $ A_{m}^{k_l}$ is a $2$-Lipschitz graph in the direction $v$.

Next pick such subsequence and such $v$ and denote $B_l\coloneqq A_m^{k_l}$.
We may assume that $v=(0,1)$, which in other words means that there are intervals $[a_l,b_l]$ and $2$-Lipschitz functions $f_l:[a_l,b_l]\to\er$ such that $B_l=\graph f_l$, $l\in\en$.
We may also assume that $f_l\geq 0$ for every $l$.

Since $\gamma_m$ is a $\C^2$ curve we know that each $f_l$ is also $\C^2$.
Moreover,
\begin{equation*}
L\geq \widetilde\C^{var}_0(M,\graph f_i)
=\int_{a_i}^{b_i}\frac{|f''|}{\left(1+(f')^2\right)^{\frac{3}{2}}}
\geq \int_{a_i}^{b_i}\frac{|f''|}{\left(1+2^2\right)^{\frac{3}{2}}}.
\end{equation*}
and so the sequence $f_l$ satisfies the assumptions of Lemma~\ref{L:limitOfGraphs}. Since $A_m$ is the limit of the sequence $\{\graph f_l\}$ in the Hausdorff distance, we obtain that $A_m$ is a DC graph, which concludes the proof of the theorem.
\end{proof}\color{black}

\color{black}
\begin{cor}
	Let $M$ be a compact subanalytic set in $\er^2$. Then $M$ is $\UWDC$.
\end{cor}

\begin{proof}
	Since $M\in\N(\er^2)$ by Proposition~\ref{prop:subanalytic} we can just apply Proposition~\ref{prop:Ner2SubsetUWDC}.
\end{proof}

%


\begin{lem}\label{L:lipschitzIsWDC}
	Let $M\subset\er^2$ be a compact Lipschitz domain. Suppose that $M$ admits the normal cycle and that $M$ satisfies condition (I).
	Then $M$ is $\UWDC$.
\end{lem}

\begin{proof}
	By Lemma~\ref{lem:LipsechitzParts}~(a) we can find pairwise disjoint Jordan curves $\gamma_i:[0,1]\to\er^2$, $i=1,\dots,j$ such that $\partial M=\bigcup_{i}\gamma_i$.
	Pick $i\in \{1,\dots,j\}$.
	We want to prove that $\gamma_i$ has finite turn.
	We start by proving the following claim.
	
	{\it Claim:} suppose that $0\leq s < t < u < 1$.
	Put $A\coloneqq\gamma_i(t)$, $A^-\coloneqq\gamma_i(s)$ and $A^+\coloneqq \gamma_i(u)$.
	Then 
	\begin{equation*}
	C^{var}_0(M, \gamma_i(s,u))\geq \rho(\alpha,\beta)
	\end{equation*}
	where
	\begin{equation*}
	\alpha = \frac{A^+-A}{|A^+-A|}
	\quad\text{and}\quad 
	\beta=\frac{A-A^-}{|A-A^-|}.
	\end{equation*}
	To prove the claim pick some $0\leq s < t < u < 1$ and put $K=\co\gamma_i([s,u])$ and $T=\co \{A^-,A,A^+\}$.
	Note that $K$ is compact since $\gamma_i$ is continuous.
	First note that $\H^1(\nor(T,A))=\rho(\alpha,\beta)$ and therefore it is enough to prove that 
	\begin{equation*}
	C^{var}_0(M, \gamma_i(s,u))\geq \H^1(\nor(T,A)^\circ).
	\end{equation*}
	Choose some $v\in \nor(T,A)^\circ$.
	Put
	\begin{equation*}
	L\coloneqq \{x\in K: x\cdot v=\min\{y\cdot v: y\in K\} \}.
	\end{equation*}
	
	First note that $v\in\nor(K,x)$ for every $x\in L$.
	Moreover, $L$ is a compact line segment and $L\subset K\setminus\{A^\pm\}$ (since clearly $A\cdot v > A^\pm\cdot v$).
	Let $x$ be an endpoint of $L$.
	From the definition of $K$ we obtain that $x\in \gamma_{i}([s,u])\setminus\{A^\pm\}=\gamma_{i}((s,u))$
	and so there is $s<p<u$ such that $\gamma_i(p)=x$.
	Due to the continuity of $\gamma_i$ there is $r>0$ such that $\partial M \cap U(x,r)\subset \gamma_{i}((s,u))\subset K$.
	Also, $x\in U(x,r)\cap \partial M \cap \partial K$.
	
	To prove the claim it is now sufficient to use the fact that $M$ satisfies condition (I) together with a well known fact that
	\begin{equation*}
	C^{var}_0(M, X)\geq \H^1(\{v: \exists x\in X, \i_M(x,v)\not=0 \}).
	\end{equation*}
	
	Next pick a partition $0=x_0<x_1<\cdots<x_n=1$.
	Without any loss of generality we may assume that $n=2j$ for some $j\in\en$.
	By the claim
	\begin{equation*}
	\begin{aligned}
	\sum_{k=1}^{n-1}&\rho\left(\frac{\gamma_{i}(x_{k})-\gamma_{i}(x_{k-1})}{|\gamma_{i}(x_{k})-\gamma_{i}(x_{k-1})|},
	\frac{\gamma_{i}(x_{k+1})-\gamma_{i}(x_{k})}{|\gamma_{i}(x_{k+1})-\gamma_{i}(x_{k})|}\right)
	\leq \sum_{k=1}^{n-1}C^{var}_0(M, \gamma_{i}(x_{k-1},x_{k+1}))\\
	&\leq \sum_{k=1}^{n-1}C^{var}_0(M, \gamma_{i}(x_{k-1},x_{k+1}))\\
	&=\sum_{k=0}^{j-1}C^{var}_0(M, \gamma_{i}(x_{2k},x_{2(k+1)}))
	+\sum_{k=1}^{j-1}C^{var}_0(M, \gamma_{i}(x_{2k-1},x_{2k+1}))\\
	&\leq 2 C^{var}_0(M, \gamma_{i}(0,1))<\infty.
	\end{aligned}
	\end{equation*}
	Hence $\gamma_i$ has a finite turn and the proof is finished.
	
	
	Fix $i\in\{1,\dots,j\}$ again. 
	Since $M$ is a Lipschitz domain and $\partial M$ is compact, there	is (by Lemma~\ref{lem:LipsechitzParts}~(b)) a partition $\{0=x_0<\cdots<x_n=1\}$ of $[0,1]$ such that, denoting $\xi_k\coloneqq \gamma|_{[x_{k},x_{k+1}]}$, $\Im(\xi_k)$ is a Lipschitz graph.
	Since $\gamma_i$ has a finite turn, each $\xi_k$ has finite turn as well.
	By Lemma~\ref{L:finiteTurnIsDC} we obtain that each $\Im(\xi_k)$ is a DC graph.
	
	Hence $\partial M$ is a union of finitely many DC graphs.
	Using the fact that $M^{c}$ has only finitely many components (since $M$ admits the normal cycle) we obtain that $M$ is $\UWDC$ by Theorem~\ref{T:characterizationUWDC}.
\end{proof}

\begin{cor}
	Suppose that $M\in \MB_2$ is compact, then $M$ is $\UWDC$.
\end{cor}
\begin{proof}
Follows from Lemma~\ref{L:lipschitzIsWDC} and  Proposition~\ref{prop:innerCurvature}. 
\end{proof}

\end{document}